\documentclass[leqno,final,twoside]{article}

\usepackage[a4paper,hmargin=1.25in,vmargin=1in]{geometry}
\usepackage{fancyhdr}

\usepackage[dvips]{graphicx}
\usepackage{epstopdf}
\usepackage[caption]{subfig}
\usepackage{algorithmic}

\ifpdf
  \DeclareGraphicsExtensions{.pdf,.pdf,.png,.jpg}
\else
  \DeclareGraphicsExtensions{.pdf}
\fi
\usepackage{float}
\usepackage{amsmath,amssymb,amsthm,amsfonts}
\newtheorem{theorem}{Theorem}[section]
\newtheorem{lemma}[theorem]{Lemma}
\newtheorem{remark}[theorem]{Remark}

\newtheorem{example}[theorem]{Example}
\newtheorem{definition}{Definition}[section]

\numberwithin{theorem}{section}
\numberwithin{equation}{section}
\numberwithin{figure}{section}

\usepackage{mathrsfs}
\usepackage{bm}
\usepackage{enumerate}

\usepackage{color}
\usepackage{xcolor}
\colorlet{inlinkcolor}{green!50!black}
\colorlet{exlinkcolor}{red!50!black}
\usepackage[colorlinks = true,
  allcolors = inlinkcolor,
  urlcolor = exlinkcolor,
            ]{hyperref}
            
\usepackage{indentfirst}

\newcommand{\He}[1]{{\left\Vert{\hskip -2.7pt}\left\vert #1 \right\vert{\hskip -2.7pt}\right\Vert}}

\newcommand{\al}{\alpha}
\newcommand{\be}{\beta}
\newcommand{\de}{\delta}

\newcommand{\ep}{\varepsilon}

\newcommand{\Ga}{\Gamma}

\newcommand{\na}{\nabla}
\newcommand{\om}{\omega}
\newcommand{\Om}{\Omega}
\newcommand{\pa}{\partial}

\newcommand{\ta}{\theta}
\newcommand{\vp}{\varphi}

\newcommand{\bq}{\mathbf{q}}
\newcommand{\bV}{\mathbf{V}}
\newcommand{\bH}{\boldsymbol{H}}
\newcommand{\bn}{\boldsymbol{n}}
\newcommand{\br}{\boldsymbol{r}}
\newcommand{\brh}{\boldsymbol{\rho}}
\newcommand{\bta}{\boldsymbol{\ta}}
\newcommand{\bvp}{\boldsymbol{\vp}}
\newcommand{\E}{\mathcal{E}}
\newcommand{\T}{\mathcal{T}}
\newcommand{\tu}{\tilde u_h}
\newcommand{\tw}{\tilde w_h}
\newcommand{\tvp}{\tilde\bvp_h}

\newcommand{\norm}[1]{\left\Vert#1\right\Vert}

\newcommand{\pd}[1]{\left\langle #1\right\rangle}
\newcommand{\set}[1]{\left\{#1\right\}}
\newcommand{\av}[1]{\left\{#1\right\}}
\newcommand{\Jm}[1]{\left[{\hskip -2.6pt}\left[#1\right]{\hskip -2.6pt}\right]}
\newcommand{\jm}[1]{[\hskip -1pt[#1]\hskip -1pt]}
\newcommand{\eq}[1]{\begin{align}#1\end{align}}
\newcommand{\eqn}[1]{\begin{align*}#1\end{align*}}
\newcommand{\nn}{\nonumber}

\title{Hybridizable Discontinuous Galerkin Methods for Helmholtz Equation with High Wave Number. Part I: Linear case}
\author{Bingxin Zhu\footnotemark[1] \and Haijun Wu
\thanks{Department of Mathematics, Nanjing University, Jiangsu, 210093, P.R. China ({\tt 867621163@qq.com}, {\tt hjw@nju.edu.cn}). This work was partially supported by the NSF of China under grants 11525103 and 91630309.}
}

\date{}  

\begin{document}
\maketitle
\begin{abstract}

   This paper addresses several aspects of the linear Hybridizable Discontinuous Galerkin  Method (HDG) for the Helmholtz equation with impedance boundary condition at high frequency.  First,  error estimates with explicit dependence on the wave number $k$ for the HDG approximations to the exact solution  $u$ and its negative  gradient $\bq=-\na u$ are derived. It is shown that $k\Vert u-u_h\Vert_{L^2(\Omega)}+ \Vert \bq-\bq_h\Vert_{L^2(\Omega)}=O(k^2h^2+k^4h^3)$ under the  conditions that $k^3h^2$ is sufficiently small and that the penalty parameter $\tau\eqsim k$, where $h$ is the mesh size. 
Note that the convergence order in $\bq_h$ is full and the pollution error is  $O(k^4h^3)$, which improve the existent results. Secondly, by using a standard postprocessing procedure from the HDG method for  elliptic problems,   a piecewise quadratic function $u_h^*$ is obtained  so that $k\Vert u-u_h^*\Vert_{L^2(\Omega)}=O(k^3h^3+k^4h^3)$. Note that the postprocessing procedure improves only the interpolation error (from $O(k^2h^2)$ to $O(k^3h^3)$) but leaves the pollution error $O(k^4h^3)$ unchanged.  Thirdly,  dispersion analyses and extensive numerical tests show that the pollution effect can be eliminated completely in 1D case and reduced greatly in 2D case by selecting appropriate penalty parameters.  The  preasymptotic error analysis of the higher order HDG method for the Helmholtz equation with high wave number is studied in Part II.
 \end{abstract}


 

\section{Introduction}\label{sec:Introduction}
In this paper, we consider preasymptotic error estimates of the hybridizable discontinuous Galerkin (HDG)  method for solving the Helmholtz equation with impedance boundary condition: 
\begin{align}
- \Delta u-k^2u &= f \quad \text{in }\Omega,\label{eq31} \\
\frac{\partial u} {\partial \bn}+\mathbf{i} k u &=g \quad \text{on }\Gamma, \label{eq32}
\end{align}
where $\Omega\in\mathbb{R}^d, d=1, 2, 3 $ is a convex polyhedral domain, $\Gamma:=\partial\Omega$, $k\gg 1$ is known as the wave number, $\mathbf{i} = \sqrt{-1}$ denotes the imaginary unit, and $\bn$ denotes the unit outward normal to $\partial \Omega$. 

The above Helmholtz problem is an approximation of the
acoustic scattering problem (with time dependence $e^{\i\om t}$) and the impedance boundary condition \eqref{eq32} can be regarded as the
lowest order approximation of the radiation condition (cf. \cite{em79}).
We remark that the Helmholtz problem \eqref{eq31}--\eqref{eq32} also arises in applications as a consequence of frequency domain treatment of attenuated scalar waves (cf. \cite{dss94}). 

The Helmholtz equation with large wave number is highly indefinite, which makes
the analysis of its discretizations very difficult. We refer to 
\cite[etc.]{ib95a, ihlenburg1997finite, ihlenburg98, Babuvska:lvoM:Sauter:2000, melenk2010convergence, Melenk:Sauter:2011, zhu2013preasymptotic, DuYu:WuHaijun:2014, WuHaijun:2014, Zhu:Burman:Wu:2012, FengXiaobing:WuHaijun:2009, FengXiaobing:WuHaijun:2011, mps13, FengXiaobing:XingYulong:2013,  Cessenat:Olivier:Despres:Bruno:1998, Amara:Mohamed:DjellouliRabia:Farhat:Charbel:2009, HiptmairR:MoiolaA:PerugiaI:2011, Monk:P:Wang:DaQing:1999, MelenkJensMarkus:1995, Engquist:Runborg:2003, ShenJie:WangLiLian:2007, Griesmaier:Monk:2011, Chen:Lu:Xu:2012, Cui:Zhang:2014, thompson06} 
for various discretization methods and their error analyses for  the Helmholtz equation with high wave number, including finite element methods (FEM), continuous interior penalty finite element methods (CIP-FEM), discontinuous Galerkin (DG) methods, ultra weak variational formulation, plane wave DG methods, spectral methods, HDG methods, and so on.

The HDG method was first proposed for the second order elliptic problem in mixed form \cite{Cockburn:2009, Cockburn:Bernardo:2010}, which  gives  simultaneously piecewise polynomial approximations of the original solution $u$, the negative flux $\bq$ (e.g. $\bq=-\na u$ for the Poisson equation), and their traces  on boundaries of mesh elements. Denote the HDG solutions by $u_h,\bq_h, \hat u_h$, and $\hat\bq_h$. They all converge in full-order, e.g., for the linear HDG method, it holds $\norm{u-u_h}_{L^2(\Om)}+\norm{\bq-\bq_h}_{L^2(\Om)}=O(h^2)$. Moreover the accuracies of the approximate solutions $u_h$ and $\bq_h$ can be enhanced by means of a  local postprocessing (see e.g. \cite{Cockburn:Bernardo:2010}). Another good property is that in the implementation the variables $u_h,\bq_h$, and $\hat\bq_h$ can be easily eliminated in a element-by-element fashion to give rise to a global system of equations involving only the numerical trace $\hat u_h$, and therefore, the HDG method possesses the flexibility for  approximation spaces of DG methods while avoids the drawback of large number of coupled unknowns of some other DG methods.

There have been several works on HDG methods for the Helmholtz equation with high wave number. Griesmaier and Monk \cite{Griesmaier:Monk:2011} prove the full-order convergence of the HDG method for the interior Dirichlet problem for the Helmholtz equation  under the conditions that $h$ is sufficiently small and the penalty parameter $\tau\eqsim1$, while not considering the dependence on the wave number $k$. Chen,  Lu, and Xu \cite{Chen:Lu:Xu:2012} derive $hp$ error estimates for the HDG method for the Helmholtz problem \eqref{eq31}--\eqref{eq32}. In particular, for the linear HDG method with  $\tau=\frac{1}{h}\mathbf{i}$, it is shown that $k\norm{u-u_h}_{L^2(\Om)}=O(k^2h^2+k^3h^2)$ and $\norm{\bq-\bq_h}_{L^2(\Om)}=O(kh+k^3h^2)$. Note that the convergence order of $\bq_h$ in $h$ is not full since $\bq_h$ is from the space of piecewise linear vector functions. The second term $O(k^3h^2)$ in each of the two error bounds is the so-called pollution error (see e.g. \cite{Babuvska:lvoM:Sauter:2000}), which dominates the first error bound for any $h>0$ and the second one if $k^2h\gtrsim 1$.  
Cui and Zhang \cite{Cui:Zhang:2014} prove error estimates for the HDG methods with pure imaginary penalty parameter $\tau$ but the error bounds contains terms of positive powers of $h^{-1}$, which are not optimal.

The purpose of this paper is to discuss several important aspects of the linear HDG which have been well understood for the elliptic problems but still not clear  for  the Helmholtz equation with high wave number. First, we derive the wave-number-explicit error estimates $k\Vert u-u_h\Vert_{L^2(\Omega)}+ \Vert \bq-\bq_h\Vert_{L^2(\Omega)}=O(k^2h^2+k^4h^3)$ under the  conditions that $k^3h^2$ is sufficiently small and that the penalty parameter $\tau\eqsim k$, which are full-order for both  $u_h$ and $\bq_h$ and the pollution error $O(k^4h^3)$ is better than those for the linear FEM \cite{WuHaijun:2014,DuYu:WuHaijun:2014} and the linear HDG method with $\tau\eqsim \frac{\mathbf{i}}{h}$ \cite{Chen:Lu:Xu:2012,Cui:Zhang:2014}. Secondly, we show that the standard postprocessing procedure from the HDG method for  elliptic problems produces  a piecewise quadratic function $u_h^*$ satisfying $k\Vert u-u_h^*\Vert_{L^2(\Omega)}=O(k^3h^3+k^4h^3)$. Note that the postprocessing procedure improves only the interpolation error (from $O(k^2h^2)$ to $O(k^3h^3)$) but leaves the pollution error $O(k^4h^3)$ unchanged.  Thirdly, we consider the selection of the penalty parameter in order to reduce the pollution effect. By dispersion analyses and extensive numerical tests, it is shown that the pollution effect can be eliminated completely in 1D case and reduced greatly in 2D case by selecting appropriate penalty parameters.  We would like to remark that the analyses are non-trivial. For example, in order to derive full order preasymptotic error estimates for both $u_h$ and $q_h$, we have used the modified duality argument and a special regularity estimate (see Remark~\ref{rerror}(b)) for the dual problem.  

The remainder of this paper is organized as follows. In \S \ref{HDG method},  we formulate the HDG method. In \S \ref{Elliptic projection}, we introduce some elliptic projections and derive theirs error estimates. In section \S \ref{Estimate}, we prove the preasymptotic error estimates of the HDG method for the Helmholtz problem. Then in  \S \ref{Postprocessing}, we apply the standard postprocessing  procedure to the HDG method for the Helmholtz problem and analyze the error of the postprocessing solution. In \S \ref{sec:Dispersion}, we carry out the dispersion analyses for the HDG method for the Helmholtz problems on one dimensional equidistant grids and two dimensional equilateral triangulations, respectively. In the last section, we present some numerical examples to verify our theoretical findings.

  Throughout the paper, $C$  is used to denote a generic positive constant which is independent of $h$, $k$, $f$, $g$ and the penalty parameters. We also use the shorthand notation $A \lesssim B$ and $B\lesssim A$ for the inequality $A \leq C B$ and $B\geq C A$. $A\eqsim B$ is a shorthand notation for the statement $A\lesssim B$ and $B\gtrsim  A$.  We assume that $k \gg 1$ since we are considering high-frequency problem.  

\section{HDG method}\label{HDG method}

In this section we recall the HDG method and introduce a variational formulation by eliminating the numerical traces, which will be used to derive error estimates. We first introduce some notation.  The standard space, norm, and inner product notation are adopted. Their definitions can be found in \cite{brenner2007mathematical,ciarlet2002finite}. In particular, $(\cdot,\cdot)$ and  $\norm{\cdot}$denotes the $L^2$-inner product  and $L^2$-norm on the complex-valued $L^2(\Om)$, respectively.
Let $\set{\mathcal{T}_h}$ be a family of regular and quasi-uniform triangulations of $\Omega$. Let $\E_h, \E_h^I$, and $\E_h^N$ be the set of all edges/faces of elements in $\T_h$, the inner edges/faces, and edges/faces on $\Ga$, respectively. For any $K\in\mathcal{T}_h$ and $e\in\E_h$, let $h_K:={\rm diam}\,(K)$ and $h_e:={\rm diam}\,(e)$. Denote by $h:=\max_{K\in\mathcal{T}_h} h_K$. Denote by $(\cdot,\cdot)$, $(\cdot,\cdot)_K$, and  $\pd{\cdot,\cdot}_e$ the $L^2$-inner product on $L^2(\Om)$, $L^2(K)$, and $L^2(e)$, respectively. For any $\mathcal{F}\subset\E_h$, let $\pd{\cdot,\cdot}_{\mathcal{F}}:=\sum_{e\in\mathcal{F}}\pd{\cdot,\cdot}_e$. Denote by $H^1(\T_h):=\prod_{K\in\T_h} H^1(K)$ and by $\bH^1(\T_h):=H^1(\T_h)^d$. 
For brevity, write $\norm{\cdot}:=\norm{\cdot}_{L^2(\Om)}$ and $\norm{\cdot}_{\pa\T_h}=\big(\sum_{K\in\T_h}\norm{\cdot}_{L^2(\pa K)}^2\big)^\frac12$.

\subsection{HDG formulation}
As usual, the HDG method designed by first rewritting \eqref{eq31}--\eqref{eq32} into the following first order system on $u$ and $\bq=-\nabla u$:
\begin{align}
\bq+\nabla u&=0\quad \text{in}\,\Omega,\label{eq33}\\
\quad\nabla\cdot\bq-k^2 u&=f\quad\text{in}\,\Omega,\label{eq34}\\
-\bq\cdot \bn+\mathbf{i} ku &=g\quad \text{on}\,\Gamma.\label{eq35}
\end{align}
Introduce the following approximation spaces for solving $u$, $\bq$, and the traces of $u$ on $\E_h$, respectively. 
\begin{align*}
V_h& :=\lbrace v_{h}:\; v_{h}|_{K}\in P_{1}(K),\forall  K\in\mathcal{T}_h\rbrace, \quad \bV_h:=(V_h)^d,\\
S_h& :=\lbrace \lambda_h:\; \lambda_{h}|_{e}\in P_1(e),\forall e\in\E_h\rbrace,
\end{align*}
where $P_1$ denotes the set of linear polynomials. Define the jump of a function $\vp$ on an edge/face $e\in\E_h$:
\eq{\label{Jm}
\Jm{\vp}:=
\begin{cases}
\vp|_{K_1}\cdot\bn_{K_1}+\vp|_{K_2}\cdot\bn_{K_2}&\text{if } e=K_1\cap K_2\in \E_h^I,\\
\vp\cdot\bn&\text{if } e\in \E_h^N.
\end{cases}
}
Note that $\Jm{\vp}$ is a vector if $\vp$ is scalar and vice versa.
Then the HDG method reads as \cite{Cockburn:2009, Cockburn:Bernardo:2010,Griesmaier:Monk:2011,Chen:Lu:Xu:2012}: Find $u_h\in V_h, \bq_h\in\bV_h, \hat{u}_h\in S_h$ such that
\begin{align}
(\bq_h,\br_h)_K& =(u_h,\nabla\cdot\br_h)_K-\int_{\partial K}^{}\hat{u}_h\bar\br_h\cdot \bn_K, \quad\forall\br_h\in\bV_h,\label{eq39}\\
(f, v_h)_K&=-(\bq_h,\nabla v_h)_K-k^2(u_h, v_h)_K+\int_{\partial K}^{}\hat{\bq}_h\cdot \bn_K \bar v_h, \quad\forall v_h\in V_h, \label{eq310}\\
\hat{\bq}_h&=\bq_h+\tau(u_h-\hat{u}_h)\bn_K \text{ on }\partial K,\quad \forall K\in\mathcal{T}_h,\label{eq311}\\
\Jm{\hat{\bq}_h}|_e&=0 \text{ if } e\in\mathcal{E}_h^{I};\quad(-\hat{\bq}_h\cdot \bn+\mathbf{i}k\hat{u}_h)|_e=g_h \text{ if } e\in \mathcal{E}_h^{N}.\label{eq312}
\end{align}
where $g_h$ is the $L^2$-projection of $g$ onto $S_h$ and $\tau$ is some penalty function defined on $\E_h$ which is assumed to be constant on each $e\in\E_h$. Recall that $\hat{\bq}_h$ and $\hat{u}_h$ are the so-called numerical trace, which are the approximations of $-\nabla u$ and $u$ on $\E_h$, respectively.

\subsection{A variational formulation on $u_h$ and $\bq_h$} In this subsection, for the purpose of theoretical analysis, we rewrite the HDG method into a variational  formulation on $u_h$ and $\bq_h$ by eliminating $\hat u_h$. 

Introduce the average of a function $v$ on $e\in\E_h$:
\eq{\label{av}
\av{v}:=
\begin{cases}
\frac12(v|_{K_1}+v|_{K_2})&\text{if } e=K_1\cap K_2\in \E_h^I,\\
v &\text{if } e\in \E_h^N.
\end{cases}
}
A direct calculation shows that the following ``magic formula" holds:
\begin{align}
\sum_{K\in\mathcal{T}_h}\int_{\partial K}v \bar{\br} \cdot \bn_K=\langle \Jm{v},\lbrace\br\rbrace\rangle_{\mathcal{E}_h^{I}}+\langle\lbrace v\rbrace,\Jm{\br}\rangle_{\mathcal{E}_h}.\label{eq36}
\end{align}
Denote by $\na_h$ the piecewise gradient operator, that is, $\na_h v|_K=\na(v|_K), \forall K\in\T_h$. From $\eqref{eq39}$, \eqref{eq310}, and  $\eqref{eq36}$, we conclude that
\eq{
(\bq_h,\br_h)  =(u_h,\nabla_h\cdot\br_h) 
 -\langle{\hat {u}_h,\Jm{\br_h}\rangle}_{\E_h},\label{eq313}
}
and
\begin{equation}
\begin{split}
(f, v_h) 
& =-(\bq_h,\nabla_h v_h) -k^2(u_h, v_h) +{\langle{\lbrace{\hat{\bq}_h}\rbrace,\Jm{v_h}}\rangle}_{\mathcal{E}_h^{I}}+{\langle{\Jm{\hat\bq_h},\lbrace v_h\rbrace}\rangle}_{\mathcal{E}_h^{N}},\label{eq317}
\end{split}
\end{equation}
From \eqref{eq311}--\eqref{eq312} and the definitions \eqref{Jm} and \eqref{av}, there hold on
 $e\in\mathcal{E}_h^{I}$, 
\eq{\Jm{\bq_h}+2\tau\big(\lbrace u_h\rbrace-\hat{u}_h\big)=0\quad\text{and}\quad \lbrace\hat\bq_h\rbrace=\lbrace\bq_h\rbrace+\frac{\tau}{2}\Jm{u_h},\label{eq314}
}
On $e\in\mathcal{E}_h^{N}$, we have $\mathbf{i}k\hat{u}_h-g_h=\hat\bq_h\cdot\bn=\Jm{\bq_h}+\tau \big(\lbrace u_h\rbrace-\hat{u}_h\big)$, and therefore
\begin{align}
\hat{u}_h&=\frac{1}{\tau+\mathbf{i}k}(\bq_h\cdot \bn+\tau u_h+g_h),\label{eq315}\\
\hat\bq_h\cdot \bn & =\frac{\mathbf{i}k}{\tau+\mathbf{i}k}(\bq_h\cdot \bn+\tau u_h)-\frac{\tau}{\tau+\mathbf{i}k}g_h.\label{eq319}
\end{align}
By substituting \eqref{eq314}--\eqref{eq319} into \eqref{eq313}--\eqref{eq317} we obtain
\begin{align}
(\bq_h,\br_h) =&(u_h,\nabla_h\cdot\br_h)-\Big\langle \frac{1}{2\tau}{\Jm{\bq_h}+\av{u_h},\Jm{\br_h}}\Big\rangle_{\mathcal{E}_h^{I}}\label{eq316}\\
&-\Big\langle\frac{1}{\tau+\mathbf{i}k}(\bq_h\cdot \bn
    +\tau u_h+ g_h),\Jm{\br_h}\Big\rangle_{{\mathcal{E}}_h^N},\nn\\
(f, v_h) =&-(\bq_h,\nabla_h v_h) -k^2(u_h, v_h) +\Big\langle{\lbrace\bq_h\rbrace+\frac{\tau}{2}\Jm{u_h},\Jm{v_h}}\Big\rangle_{\mathcal{E}_h^{I}}\label{eq320} \\
&+\Big\langle\frac{\mathbf{i}k}{\tau+\mathbf{i}k}(\bq_h\cdot \bn+\tau u_h),{\lbrace v_h\rbrace}\Big\rangle_{\mathcal{E}_h^{N}}-\Big\langle \frac{\tau}{\tau+\mathbf{i}k} g_h,{\lbrace v_h\rbrace}\Big\rangle_{\mathcal{E}_h^{N}}.\nn
\end{align}
Introduce the sesquilinear form on $\big(H^1(\T_h),\bH^1(\T_h)\big)\times\big(H^1(\T_h),\bH^1(\T_h)\big)$:
\begin{equation}\label{A}
\begin{split}
A(u,\bq ;v, \br):=&~(\bq, \br) -k^2(u, v) -(u,\nabla_h\cdot\br) -(\bq,\nabla_h v) +\Big\langle \frac{1}{2\tau}{\Jm{\bq}+\av{u},\Jm{\br}}\Big\rangle_{\mathcal{E}_h^{I}}\\
&+\Big\langle{\lbrace\bq\rbrace+\frac{\tau}{2}\Jm{u},\Jm{v}}\Big\rangle_{\mathcal{E}_h^{I}}+\Big\langle\frac{1}{\tau+\mathbf{i}k}(\bq\cdot \bn+\tau u),
 \br\cdot \bn-\mathbf{i}kv\Big\rangle_{\mathcal{E}_h^{N}},
\end{split}
\end{equation}

and define
\begin{align}
F(v, \br):=(f, v) -\Big\langle \frac{1}{\tau+\mathbf{i}k}g, \br\cdot\bn-\bar{\tau} v\Big\rangle_{\mathcal{E}_h^{N}}. \label{eq322}
\end{align}
From \eqref{eq316} and \eqref{eq320}, the HDG method is rewritten as: find $(u_h,\bq_h)\in (V_h,\bV_h)$, such that
\begin{align}
A(u_h,\bq_h; v_h,\br_h)=F(v_h,\br_h)\quad\forall \br_h\in\bV_h, v_h\in V_h.\label{eq350}
\end{align}

Noting from \eqref{eq36} that
\eqn{-(u,\na_h \cdot\br)&=(\na_h u,\br)-\pd{\Jm{u},\av{\br}}_{\E_h^I}-\pd{\av{u},\Jm{\br}}_{\E_h},\\
-(\bq,\na_h v)&=(\na_h\cdot\bq,v)-\pd{\av{\bq},\Jm{v}}_{\E_h^I}-\pd{\Jm{\bq},\av{v}}_{\E_h},}
we have another two equivalent forms of $A$ which are also useful in the analysis:
\begin{equation}
\begin{split}
A(u,\bq ;v, \br):=&~(\bq, \br) -k^2(u, v) +(\nabla_h u, \br) -(\bq,\nabla_h v) +\langle\lbrace\bq\rbrace,\Jm{v}\rangle_{\mathcal{E}_h^{I}}\\
&-\langle\Jm{u},\lbrace \br\rbrace\rangle_{\mathcal{E}_h^{I}}
+\Big\langle\frac{1}{2\tau}\Jm{\bq},\Jm{\br}\Big\rangle_{\mathcal{E}_h^{I}}+\Big\langle\frac{\tau}{2}\Jm{u},\Jm{v}\Big\rangle_{\mathcal{E}_h^{I}}+\mathbf{i}k\langle \lbrace u\rbrace,\lbrace v\rbrace\rangle_{\mathcal{E}_h^{N}}\\
&+\Big\langle\frac{1}{\tau+\mathbf{i}k}(\bq\cdot \bn-\mathbf{i}ku),
 \br\cdot \bn-\mathbf{i}kv\Big\rangle_{\mathcal{E}_h^{N}},\label{A2}
\end{split}
\end{equation}
and
\begin{equation}
\begin{split}
A(u,\bq ;v, \br):=&~(\bq, \br) -k^2(u, v) +(\nabla_h u, \br) +(\na_h\cdot\bq,v) +\Big\langle\Jm{\bq},\frac{1}{2\bar\tau}\Jm{\br}-\av{v}\Big\rangle_{\mathcal{E}_h^{I}}\\
&+\Big\langle\Jm{u},\frac{\bar\tau}{2}\Jm{v}-\lbrace \br\rbrace\Big\rangle_{\mathcal{E}_h^{I}}+\Big\langle\bq\cdot \bn-\mathbf{i}ku,
 \frac{1}{\bar\tau-\mathbf{i}k}(\br\cdot \bn-\bar\tau v)\Big\rangle_{\mathcal{E}_h^{N}},\label{A3}
\end{split}
\end{equation}

If $u$ is the solution to the Helmholtz problem\eqref{eq31}--\eqref{eq32} and $\bq=-\na u$, by using \eqref{A3} it easy to verify that
\eq{\label{VP1}
A(u,\bq;v,\br)=F(v,\br),\quad\forall v\in H^1(\T_h), \br\in\bH^1(\T_h).}
That is, the HDG formula \eqref{eq350} is consistent with the Helmholtz problem. As a consequence, we have the following Galerkin orthogonality.
\begin{align}
A(u-u_h,\bq-\bq_h; v_h,\br_h)=0\quad \forall ~v_h\in V_h,\br_h\in \bV_h.\label{eq3133}
\end{align}

For the ease of presentation, we assume that the penalty parameter $\tau$ is a positive constant in the error analysis since it can be easily extended to the case of a complex number with positive real and imaginary parts. 

 Using \eqref{A2}, we introduce the following norm on $\big(H^1(\T_h),\bH^1(\T_h)\big)$:
\begin{equation}
\begin{split}
\He{v,\br}^2:=&\Re \big(A(v, \br;v, \br)\big)+k^2(v,v)\\
= &~\norm{\br}^2 +\frac{1}{2\tau}\norm{\Jm{\br}}_{\mathcal{E}_h^{I}}^2+\frac{\tau}{2}\norm{\Jm{v}}_{\mathcal{E}_h^{I}}^2
+\frac{\tau}{\tau^2+k^2}\norm{\br\cdot \bn-\mathbf{i}kv}_{\mathcal{E}_h^{N}}^2.\label{norm}
\end{split}
\end{equation}

\section{Elliptic projections}\label{Elliptic projection}
In this section, we derive error estimates of some elliptic projections which will be used in the modified duality argument in the error analysis for the HDG methods for the Helmholtz problem.

We first recall the HDG projection introduced in \cite{Cockburn:Bernardo:2010}. Given a number $\be\neq 0$, $\Pi^\be(u,\bq)=(\Pi_1^\be u,\Pi_2^\be\bq)\in (V_h,\bV_h)$ is defined as follows: For any $K\in\T_h, e\subset\pa K$,  
\begin{align}
(\Pi_1^\beta u ,v_h)_K& =(u, v_h)_K\quad\forall\, v_h\in P_0,\label{P1}\\
(\Pi_2^\beta\bq,\br_h)_K& =(\bq,\br_h)_K\quad\forall\,\br_h\in(P_0)^d,\label{P2} \\
\langle{\Pi_2^\beta\bq\cdot \bn_K}+\beta\Pi_1^\beta u,\mu\rangle_e& =\langle{\bq\cdot \bn_K}+\beta u,\mu\rangle_e\quad\forall\,\mu\in P_1. \label{P3}
\end{align}
Here $P_0$ denotes the space of constant functions. We have the following estimates for the projection $\Pi^\beta$ \cite[Theorem 2.1]{Cockburn:Bernardo:2010}:
\begin{lemma}\label{lem:Pi}
 Let $1\leq s$,$t \leq 2$. Then
 \begin{align}
{\Vert {\Pi_1^{\beta} u-u}\Vert}_{L^2(K)}&\lesssim h_K^{s} \vert u\vert_{H^{s}(K)}+\vert\beta\vert^{-1}h_K^{t}\vert\nabla\cdot\bq\vert_{H^{t-1}(K)},\label{eq324}\\
{\Vert\Pi_2^{\beta}\bq-\bq\Vert}_{L^2(K)}& \lesssim\vert\beta\vert h_K^{s}\vert u\vert_{H^{s}(K)}+h_K^{t}\vert\bq\vert_{H^t(K)},\quad \forall K\in{\mathcal{T}_h}.\label{eq325}
\end{align}
\end{lemma}
For simplicity, denote by $\Pi:=\Pi^\tau$ and by $\Pi^*:=\Pi^{-\tau}$.
We introduce the following definite sesquilinear form by removing the indefinite term $-k^2(u,v)$ in the sesquilinear form $A$ for the Helmholtz problem (see \eqref{A} or \eqref{A2}):
\eq{\label{hA}
\hat A(u,\bq ;v, \br):=A(u,\bq ;v, \br)+k^2(u,v).}
Given $w\in H^2(\Om)$ and $\bvp=\na w$, their elliptic projections $\tw\in V_h$ and $\tvp\in\bV_h$ are define by
\eq{\label{EP}
\hat A(v_h,\br_h; \tw,\tvp)=\hat A(v_h,\br_h; w,\bvp),\quad\forall v_h\in V_h, \br_h\in\bV_h.}
Note that $\tw$ and $\tvp$ are the HDG approximations of the following elliptic problem:
\eqn{
\bvp-\nabla w& =0,\quad-\nabla\cdot\bvp =\tilde F, \quad\text{in} \ \Omega,\\
\bvp\cdot \bn-\mathbf{i}k w& =\tilde g\quad \text{on} \ \Gamma,}
for some functions $\tilde F$ and $\tilde g$. As usual, decompose the errors as:
\begin{align*}
w-\tw&=w-\Pi_1^*w-(\tw-\Pi_1^*w)=:\tilde{\xi}-\tilde\om_h,\\
\bvp-\tvp&=\bvp-\Pi_2^*\bvp-(\tvp-\Pi_2^*\bvp)=:\tilde{\brh}-\tilde{\bta}_h. 
\end{align*}
The following theorem gives error estimates of the elliptic projections.
\begin{theorem}\label{theo41} For any $w\in H^2(\Om)$ and $\bvp=\na w$, let $\tw$ and $\tvp$ be the elliptic projections defined in \eqref{EP}. Then
\begin{align}
\norm{\bvp-\tvp} &\lesssim \norm{\bvp-\Pi_2^*\bvp},\label{eqEP1}\\
\norm{w-\tw} &\lesssim \Vert w-\Pi_1^*w\Vert+(h+\tau h^2)\norm{\bvp-\Pi_2^*\bvp} \label{eqEP2}.
\end{align}
\end{theorem}
\begin{proof} From the definition of $\Pi^*=\Pi^{-\tau}$ (see \eqref{P1}--\eqref{P3} with $\be=-\tau$), we conclude that
\eq{\label{tPi}
\begin{split}
(\tilde\xi, v_h)_K=0, &\quad\forall\, v_h\in P_0,\\
(\tilde\brh,\br_h)_K =0,&\quad\forall\,\br_h\in(P_0)^d, \quad\forall K\in\T_h,\\
\langle\Jm{\tilde\brh}-2\tau \{\tilde\xi\},\mu\rangle_e=0,&\quad \langle2\av{\tilde\brh}-\tau \jm{\tilde\xi},\mu\bn_e\rangle_e=0,\quad\forall\,\mu\in P_1,  e\in\E_h^I, \\
\langle\tilde\brh\cdot\bn-\tau \tilde\xi,\mu\rangle_e=0, &\quad\forall\,\mu\in P_1,  e\in\E_h^N. 
\end{split}
}
Therefore, it follows from \eqref{EP}, \eqref{hA}, and \eqref{A3} that 
\eq{\label{eq45}
\hat{A}(v_h,\br_h;\tilde\om_h,\tilde\bta_h)=&\hat{A}(v_h,\br_h;\tilde{\xi},\tilde{\brh})=(\br_h, \tilde\brh),\quad\forall v_h\in V_h,\br_h\in\bV_h.}
By taking $v_h=\tilde\om_h$ and $\br_h=\tilde\bta_h$ in \eqref{eq45} and using \eqref{norm}, we conclude that
\begin{align*}
\He{\tilde\om_h,\tilde{\bta}_h}^2=\Re\big(\hat{A}(\tilde\om_h,\tilde{\bta}_h;\tilde\om_h,\tilde{\bta}_h)\big)=\Re\big((\tilde{\bta}_h,\tilde{\brh}) \big)\le\|\tilde{\brh}\|\|\tilde{\bta}_h\| .
\end{align*}
and hence
\begin{align}
\Vert\tilde{\bta}_h\Vert\lesssim\Vert\tilde{\brh}\Vert,\label{eq46}
\end{align}
which implies \eqref{eqEP1}.

Next we turn to prove \eqref{eqEP2} by using the duality argument. Introduce the dual problem
\begin{align}
\mathbf{Q}+\nabla U& =0,\quad \nabla\cdot\mathbf{Q} =\tilde\om_h,\quad\text{in } \, \Omega,\label{eq47}\\
\mathbf{Q}\cdot \bn+\mathbf{i}kU& =0\quad \text{on } \, \Gamma.\label{eq49}
\end{align}
The regularity theory of elliptic equations says that
\begin{align}\label{hatreg} 
\Vert  U\Vert_{H^2(\Omega)}+\Vert\mathbf{Q}\Vert_{H^1(\Omega)}\lesssim\norm{\tilde\om_h}.
\end{align}
Similar to \eqref{VP1} we have
\eq{\label{VP2}
\hat A(U,\mathbf{Q};v,\br)=(\tilde\om_h, v),\quad\forall v\in H^1(\T_h), \br\in\bH^1(\T_h).}

Denote by 
$$\hat{\xi}:=U-\Pi_1 U,\quad \hat{\brh}:=\mathbf{Q}-\Pi_2\mathbf{Q}.$$
Similar to \eqref{tPi}, from \eqref{P1}--\eqref{P3} with $\be=\tau$
, we have 
\eq{\label{hPi}
\begin{split}
(\hat{\xi}, v_h)_K=0, &\quad\forall\, v_h\in P_0,\\
(\hat{\brh},\br_h)_K =0,&\quad\forall\,\br_h\in(P_0)^d, \quad\forall K\in\T_h,\\
\langle\Jm{\hat{\brh}}+2\tau \{\hat{\xi}\},\mu\rangle_e=0,&\quad \langle2\av{\hat{\brh}}+\tau \jm{\hat{\xi}},\mu\bn_e\rangle_e=0,\quad\forall\,\mu\in P_1,  e\in\E_h^I, \\
\langle\hat{\brh}\cdot\bn+\tau \hat{\xi},\mu\rangle_e=0, &\quad\forall\,\mu\in P_1,  e\in\E_h^N. 
\end{split}}
It follows from \eqref{hA} and \eqref{A} that 

\eq{\label{eq45a}
\hat{A}(\hat{\xi},\hat{\brh};v_h,\br_h)=(\hat{\brh},\br_h),\quad\forall v_h\in V_h,\br_h\in\bV_h.}
By taking $v_h=\tilde\om_h, \br_h=\tilde\bta_h$ in \eqref{VP2} and using \eqref{eq45}, \eqref{eq45a}, and \eqref{tPi}, we conclude that
\eqn{
\norm{\tilde\om_h}^2&=\hat{A}(U,\mathbf{Q};\tilde\om_h,\tilde\bta_h)=\hat{A}(\Pi_1U,\Pi_2\mathbf{Q};\tilde\om_h,\tilde\bta_h)+\hat{A}(\hat{\xi},\hat{\brh};\tilde\om_h,\tilde\bta_h)\\
&=(\Pi_2\mathbf{Q}, \tilde\brh)+(\hat\brh,\tilde\bta_h)\\
&=-(\na U, \tilde\brh)-(\hat\brh,\tilde\brh)+(\hat\brh,\tilde\bta_h)\\
&=(\na(I_hU -U), \tilde\brh)-(\hat\brh,\tilde\brh)+(\hat\brh,\tilde\bta_h),}
where $I_h$ is the finite element interpolation operator. Therefore, from Lemma~\ref{lem:Pi}, \eqref{hatreg}, and \eqref{eq46} we have
\eqn{
\norm{\tilde\om_h}^2\lesssim& h\norm{\tilde\brh}|U|_{H^2(\Om)}
+\big(\tau h^2\vert U\vert_{H^2(\Om)}+h\vert\mathbf{Q}\vert_{H^1(\Om)}\big)\norm{\tilde\brh}\\
\lesssim& (\tau h^2+h)\norm{\tilde\om_h}\norm{\tilde\brh},}
which implies that \eqref{eqEP2} holds. This completes the proof of the theorem.
\end{proof}

\section{Error estimates for the HDG method}\label{Estimate}
In this section, we derive preasymptotic error estimates of the HDG solutions for the Helmholtz problem. 

We first recall wave-number-explicit stability and regularity estimate for the Helmholtz problem \eqref{eq31}--\eqref{eq32} (see \cite{Melenk:J:1995,Hetman:Ulrich:2007,Cummings:Peter:Feng:2006}).
\begin{lemma}\label{lemmareg}  Let $u$ be the solution to \eqref{eq31}--\eqref{eq32}. Then
\begin{align*}
k\Vert u\Vert+\Vert u\Vert_{H^1(\Omega)}&\lesssim \Vert f\Vert+\Vert g\Vert_{L^2(\Gamma)},\\
\Vert u\Vert_{H^2(\Omega)}&\lesssim k\big(\Vert f\Vert+\Vert g\Vert_{L^2(\Gamma)}+k^{-1}\Vert g\Vert_{H^{1/2}(\Gamma)}\big).
\end{align*}
\end{lemma}
\begin{remark} If $\Om$ is strictly star-shaped and sufficiently smooth, say $C^{2,1}$, then 
\eqn{\Vert u\Vert_{H^3(\Omega)}\lesssim k^2\big(\Vert f\Vert+\Vert g\Vert_{L^2(\Gamma)}+k^{-1}\Vert g\Vert_{H^{1/2}(\Gamma)}+k^{-2}(\norm{f}_{H^1(\Om)}+\Vert g\Vert_{H^{3/2}(\Gamma)})\big).}
In the following, we denote by $C_u:=k^{-1}|u|_{H^2(\Om)}+k^{-2}|u|_{H^3(\Om)}$ which is expect to be independent of the wave number $k$.
\end{remark}

 Similar to the error estimates for elliptic projections, we decompose the error as follows:
\begin{align*}
u-u_h &=(u-\Pi_1u)-(u_h-\Pi_1u)=:\xi-\om_h,\\
\bq-\bq_h &=(\bq-\Pi_2\bq)-(\bq_h-\Pi_2\bq)=:\brh-\bta_h.
\end{align*}
As \eqref{tPi}, from the definition \eqref{P1}--\eqref{P3} with $\be=\tau$, we have 
\eq{\label{Pi}
\begin{split}
(\xi, v_h)_K=0, &\quad\forall\, v_h\in P_0,\\
(\brh,\br_h)_K =0,&\quad\forall\,\br_h\in(P_0)^d, \quad\forall K\in\T_h,\\
\langle\Jm{\brh}+2\tau \{\xi\},\mu\rangle_e=0,&\quad \langle2\av{\brh}+\tau\jm{\xi},\mu\bn_e\rangle_e=0,\quad\forall\,\mu\in P_1,  e\in\E_h^I, \\
\langle\brh\cdot\bn+\tau \xi,\mu\rangle_e=0, &\quad\forall\,\mu\in P_1,  e\in\E_h^N. 
\end{split}
}

The following theorem gives the error estimates of the linear HDG solution $u_h$ and $\bq_h$.
\begin{theorem}\label{thm:erresti}
Let $(u,\bq)$ and $(u_h,\bq_h)$ be the solutions of \eqref{eq33}--\eqref{eq35} and \eqref{eq39}--\eqref{eq310}, respectively. Suppose $\tau\eqsim k $. Then there exists a positive constant $C_0<1$ independent of $k,h$ and $\tau$, such that if $k^3h^2 \leq C_0$, 
\begin{align}
k\Vert u-u_h\Vert_{L^2(\Omega)}&\lesssim (k+k^3h)\Vert\xi\Vert+k^2h\norm{\brh}\lesssim (k^2h^2+k^4h^3)C_u,\label{eu}\\
 \Vert\bq-\bq_h\Vert_{L^2(\Omega)}&\lesssim (k+k^3h)\Vert\xi\Vert+(1+k^2h)\norm{\brh}\lesssim (k^2h^2+k^4h^3)C_u.\label{eq}
 \end{align}
\end{theorem}
\begin{proof}
First, from  \eqref{eq3133}, \eqref{A}, and \eqref{Pi}, we conclude that
\begin{align}\label{eq333}
A(\om_h,\bta_h; v_h,\br_h)=A(\xi,\brh; v_h,\br_h)
=(\brh,\br_h) -k^2(\xi, v_h),\quad\forall v_h\in V_h,\br_h\in\bV_h,
\end{align}
and hence from \eqref{norm},
\begin{align*}
\norm{\bta_h}^2\le\He{ \om_h,\bta_h}^2=\Re\left((\brh,\bta_h) -k^2(\xi, \om_h) \right)+k^2(\om_h, \om_h),
\end{align*}
which implies by the Cauchy's inequality and Young's inequality that
\eq{\label{eq234}
\norm{\bta_h}\lesssim \norm{\brh}+k\norm{\xi}+k\norm{\om_h}.}

Next we derive the estimate of $\norm{\om_h}$ by the modified duality argument which use the elliptic projections of the solutions of the dual problem instead of the interpolations of them used in the traditional duality argument.
Introduce the dual problem:
\begin{align}
\bvp-\nabla w& =0, \quad -\nabla\cdot\bvp-k^2w =\om_h,\quad \text{in} \ \Omega,\label{eq414}\\
\bvp\cdot \bn-\mathbf{i}kw& =0,\quad \text{on} \ \Gamma.\label{eq416}
\end{align}
Similar to Lemma~\ref{lemmareg}, we have
 \begin{align}
k^2\Vert w\Vert+k\Vert w\Vert_{H^1(\Omega)}+\norm{w}_{H^2(\Om)}+\norm{\bvp}_{H^1(\Om)}&\lesssim k\Vert \om_h\Vert, \label{eq417}
\end{align}
and from the inverse inequality
\eq{\label{eq418}\norm{\na\cdot\bvp}_{H^1(\Om)}=\norm{k^2w+\om_h}_{H^1(\Om)}\lesssim(k^2+h^{-1})\norm{\om_h}.}
From \eqref{A} and \eqref{eq414}--\eqref{eq416}, we have
\eqn{
A(v_h,\br_h; w,\bvp)&=(\br_h,\bvp)-(v_h,k^2w+\na\cdot\bvp)-(\br_h,\na w)\nn\\
&=(v_h, \om_h) ,\quad \forall\, v_h\in V_h,\br_h\in\bV_h.}
Therefore, it follows from the definition of elliptic projections $\tw$ and $\tvp$ in \eqref{EP}, \eqref{hA}, \eqref{eq333}, that
\begin{align*}
{\Vert \om_h\Vert}^2=&A(\om_h,\bta_h; w,\bvp)\\
=&A(\om_h,\bta_h; w-\tilde{w}_h,\bvp-\tilde{\bvp}_h)+A(\om_h,\bta_h;\tilde{w}_h,\tilde{\bvp}_h)\\
=&-k^2(\om_h, w-\tilde{w}_h) +(\brh,\tilde{\bvp}_h) -k^2(\xi,\tilde{w}_h) \\
=&-k^2(\om_h, w-\tilde{w}_h) +(\brh,\tilde{\bvp}_h-\bvp) +(\brh,\na w) -k^2(\xi,\tilde{w}_h-w)-k^2(\xi, w).
\end{align*}
Let $Q_0w$ be the $L^2$-projection of $w$ onto the piecewise constant space $\prod_{K\in\mathcal{T}_h}P_0(K)$. From the othogonalities in \eqref{Pi} we have
\begin{align*}
{\Vert \om_h\Vert}^2
=&-k^2(\om_h, w-\tilde{w}_h) +(\brh,\tilde{\bvp}_h-\bvp)  -k^2(\xi,\tilde{w}_h-w) \\
&+\left(\brh, \nabla (w-I_hw)\right)-k^2(\xi, w-Q_0w).
\end{align*}
Suppose $\tau h\lesssim 1$. Then using Theorem \ref{theo41} and Lemma \ref{lem:Pi}, we obtain
\begin{align*}
{\Vert \om_h\Vert}^2\leq& k^2\big(\Vert \om_h\Vert+\Vert \xi\Vert\big)\Vert w-\tilde{w}_h\Vert+\Vert\brh\Vert\Vert \bvp -\tilde{\bvp}_h\Vert\\
&+\Vert\brh\Vert\Vert \nabla(w-I_hw)\Vert+k^2\Vert \xi\Vert\Vert w-Q_0w\Vert\\
\lesssim& k^2\big(\Vert \om_h\Vert+\Vert \xi\Vert\big)\bigl(\Vert w-\Pi_1^*w\Vert+h\Vert\bvp-\Pi_2^*\bvp\Vert\bigr)\\
&+\Vert\brh\Vert\bigl(\Vert \bvp -\Pi_2^*\bvp \Vert+h\vert w\vert_{H^2(\Omega)}\bigr)+k^2h\Vert \xi\Vert\vert w\vert_{H^1(\Omega)}\\
\lesssim&k^2\big(\Vert \om_h\Vert+\Vert \xi\Vert\big)\bigl((h^2+\tau h^3)\vert w\vert_{H^2(\Omega)}+\tau^{-1}h^2\vert\nabla\cdot\bvp\vert_{H^1(\Om)}+h^2\vert\bvp\vert_{H^1(\Omega)}\bigr)\\
&+\Vert\brh\Vert\bigl((h+\tau h^2)\vert w\vert_{H^2(\Omega)}+h\vert\bvp\vert_{H^1(\Omega)}\bigr)+k^2h\Vert \xi\Vert\vert w\vert_{H^1(\Omega)},
\end{align*}
which together with \eqref{eq417}--\eqref{eq418} implies that
\begin{align}\label{eo}
\Vert \om_h\Vert\lesssim&k^2\big(\Vert \om_h\Vert+\Vert \xi\Vert\big)\bigl(kh^2+\tau^{-1}h^2(k^2+h^{-1})\bigr)+\Vert\brh\Vert\bigl(kh\bigr)+k^2h\Vert \xi\Vert\\
\lesssim&(k^3h^2+kh)\Vert \om_h\Vert+(k^3h^2+kh+k^2h)\Vert\xi\Vert+kh\norm{\brh}.\nn
\end{align}
Clearly, there exists a constant $C_0>0$ independent of $k$ and $h$ such that if $k^3h^2<C_0$, then
\begin{align}\label{eom}
\Vert \om_h\Vert\lesssim&(kh+k^2h)\Vert\xi\Vert+kh\norm{\brh},
\end{align}
which together with Lemma~\ref{lem:Pi} implies that \eqref{eu} holds. And \eqref{eq} follows by substituting \eqref{eom} into \eqref{eq234} and using the triangle inequality.
This completes the proof of Theorem \ref{thm:erresti}. 
\end{proof}
\begin{remark}\label{rerror} {\rm (a)} The error bounds in the theorem consist of two parts, the interpolation error (or HDG projection error) $O(k^2h^2)$ and the pollution error $O(k^4h^3)$. The results in \cite{Chen:Lu:Xu:2012} show that  $k\norm{u-u_h}_{L^2}=O(k^2h^2+k^3h^2)$ and $\norm{\bq-\bq_h}_{L^2}=O(kh+k^3h^2)$ if $\tau=\frac{1}{h}\mathbf{i}$. Our results say that taking $\tau\eqsim k$ improves both the pollution error (to $O(k^4h^3)$) and the error of $\bq_h$ (to full order in $h$).

{\rm (b)} The trick of using the regularity estimate \eqref{eq418} in \eqref{eo} to derive \eqref{eom} is crucial for the proof of the theorem. Otherwise, if the usual regularity estimate $\norm{\na\cdot\bvp}\le\norm{w}_{H^2(\Om)}\lesssim k\norm{\om_h}$ was used instead of using \eqref{eq418}, it would  be 
required that $k^2h$ is sufficiently small.

{\rm (c)} If $\tau=\mathbf{i}\al$ for some $\al\eqsim k$, the error estimates \eqref{eu}--\eqref{eq} still hold but without any mesh constraint (that is, the condition $k^3h^2\le C_0$ can be removed). While our analysis for the real $\tau$ is still meaningful since our dispersion analysis in \S\ref{sec:Dispersion} shows that the optimal penalty parameter is usually a real number. The error estimates for the case of real $\tau$ is still open when the mesh condition $k^3h^2\le C_0$ is not satisfied.

{\rm (d)} Our dispersion analysis and numerical tests in the last two sections indicate that the pollution errors may be eliminated in 1D and greatly reduced in higher dimensions by tuning the penalty parameter $\tau$. 

{\rm (e)} The estimate \eqref{eom} says that the error between $u_h$ and $\Pi_1u$ is superconvergent in $h$, which will be used to do the postprocessing in the next section.

{\rm (f)} In the part II of this series, the following error estimates will be derived for the $p^{\rm th}$ order HDG methods under the mesh condition $k^{2p+1}h^{2p}\le C_0$. 
\eqn{k\Vert u-u_h\Vert+\Vert \bq-\bq_h\Vert\lesssim (kh)^{p+1}+k(kh)^{2p+1}.}
\end{remark}

\section{Postprocessing}\label{Postprocessing}
In this section, we apply the standard postprocessing for the HDG methods for elliptic problems (see e.g. \cite{Cockburn:Bernardo:2010}) to our case of Helmholtz equations with high wave numbers. 

 Denote by $m_K v=\frac{1}{\vert K\vert}\int_{K}vdx$ the integral average of a function $v$ on $K$. Clearly, $m_K\Pi_1u=m_Ku$ and as a consequence of \eqref{eom}, $m_Ku_h$ is superclose to $m_Ku$.  Define $P_2^{0}(K):=\lbrace v\in P_2(K):\;m_Kv=0\rbrace$. The postprocessing solution $u_h^*$ is defined by $u_h^{*}|_K\in P_2(K)$ satisfying 
\begin{gather}
\begin{cases}
(\nabla u_h^*,\nabla v)_K-k^2(u_h^*,v)_K=(f, v)_K-\int_{\partial K}\hat{\bq}_h\cdot n_K\bar{v},& \forall v\in P_2^0(K),\\
m_Ku_h^{*}=m_Ku_h,& \forall K\in\mathcal{T}_h.
\end{cases}\label{eq551}
\end{gather}
Let $V_{h,2}:=\prod_{K\in\mathcal{T}_h}P_2(K)$, the following theorem gives the error estimate of $u_h^{*}$.

\begin{theorem}\label{theo51} Choose $g_h$ such that $g_h|_e$ is the $L^2$-projection of $g|_e$ onto $P_2(e)$ for any $e\in\E_h^N$. 
Assume that $\tau\eqsim k$.  Then there exists a positive constant $C_0$ independent of $k$ and $h$ such that the following estimate holds under the mesh condition $k^3h^2\le C_0$.
\begin{align}\label{errpost}
k\Vert u-u_h^*\Vert\lesssim (k^3h^3+k^4h^3)C_u.
\end{align}
\end{theorem}
\begin{proof} Following the proof of \cite[Theorem 2.2]{stenberg1991postprocessing}, 
let $\tu\in V_{h,2}$ be the $L^2$-projection of $u$ to $V_{h,2}$ and take $v=(I-m_K)(\tu-u_h^{*})\in P_2^0(K)$ in \eqref{eq551} to obtain
\begin{align*}
(\nabla v,\nabla v)_K&-k^2(v, v)_K=(\nabla(\tu-u_h^{*}),\nabla v)_K-k^2(\tu-u_h^{*}, v)_K\\
=&(\nabla (\tu-u),\nabla v)_K-k^2(\tu-u, v)_K-\int_{\partial K}(\bq-\hat{\bq}_h)\cdot \bn_K \bar{v} 
\end{align*}
Clearly,
\eqn{\int_{\partial K}(\bq-\bq_h)\cdot \bn_K \bar{v}&=(\nabla\cdot(\bq-\bq_h),v)_K+(\bq-\bq_h,\nabla v)_K\\
&=(\nabla\cdot(\bq-\br_h),v)_K+(\bq-\bq_h,\nabla v)_K,\quad\forall \br_h\in\bV_h.}
We have
\begin{align}\label{ev1}
(\nabla v,\nabla v)_K
=&(\nabla (\tu-u),\nabla v)_K-k^2(\tu-u, v)_K-\int_{\partial K}(\bq_h-\hat{\bq}_h)\cdot \bn_K \bar{v}  \\
&-(\nabla\cdot(\bq-\br_h),v)_K-(\bq-\bq_h,\nabla v)_K+k^2(v, v)_K.\nn
\end{align}
Next we estimate the first three terms on the right hand side of \eqref{ev1}. For any $ v_h\in V_{h,2}$, we have
\begin{align}
\Vert u-\tu\Vert_{L^2(K)}&\leq\Vert u-v_h\Vert_{L^2(K)},\label{eq565a}\\
\Vert \nabla(u-\tu)\Vert_{L^2(K)}&\leq\Vert \nabla(u-v_h)\Vert_{L^2(K)}+\Vert\nabla(v_h-\tu)\Vert_{L^2(K)}\label{eq565}\\
&\lesssim\Vert\nabla(u-v_h)\Vert_{L^2(K)}+h_K^{-1}\Vert v_h-\tu\Vert_{L^2(K)}\nn\\
&\lesssim \Vert\nabla(u-v_h)\Vert_{L^2(K)}+h_K^{-1}\Vert u-v_h\Vert_{L^2(K)}.\nn
\end{align}
For $e\subset\pa K$ with $e\in\mathcal{E}_h^{I}$, it follows from \eqref{eq314} that
\begin{align}\label{eq561}
\big|(\bq_h-\hat{\bq}_h)\cdot \bn_K\big|&\le\big|(\bq_h-\lbrace \bq_h\rbrace)\cdot \bn_K\big|+\frac{\tau}{2}\big|\Jm{u_h}\big|\\
&=\frac{1}{2}\big|\Jm{\bq_h-\bq}\big|+\frac{\tau}{2}\big|\Jm{u_h-u}\big|.\nn
\end{align}
For $e\subset\pa K$ with $e\in\mathcal{E}_h^N$, it follows from \eqref{eq319} and \eqref{eq36} that
\begin{align}\label{eq562}
\bigg|\int_e(\bq_h-\hat{\bq}_h)\cdot \bn_K \bar{v}\bigg|
&=\bigg|\int_e\Big(\frac{\tau}{\tau+\mathbf{i}k}(\bq_h\cdot \bn_K+g_h)-\frac{\mathbf{i}k\tau}{\tau+\mathbf{i}k}u_h\Big)\bar{v}\bigg|\\
&=\bigg|\int_e\Big(\frac{\tau}{\tau+\mathbf{i}k}(\bq_h-\bq)\cdot \bn_K-\frac{\mathbf{i}k\tau}{\tau+\mathbf{i}k}(u_h-u)\Big)\bar{v}\bigg|.\nn
\end{align}
Since $m_Kv=0$, from the local Poincar\'e inequality, the local trace inequality and  the inverse inequality, we conclude that 
\begin{align}
\Vert v\Vert_{L^2(K)}\lesssim h_K\Vert \nabla v\Vert_{L^2(K)},\quad \Vert v\Vert_{L^2(\partial K)}\lesssim h_K^{\frac{1}{2}}\Vert \nabla v\Vert_{L^2(K)}.\label{eq53}
\end{align}
By using \eqref{ev1}, the Cauchy's inequality, and \eqref{eq565a}--\eqref{eq53}, we have
\begin{align}\label{ev}
\Vert\nabla v\Vert_{L^2(K)}\lesssim&~\Vert\nabla(u-v_h)\Vert_{L^2(K)}+(h_K^{-1}+k^2h_K)\Vert u-v_h\Vert_{L^2(K)}\nn\\
&+h_K^{\frac{1}{2}}\big(\norm{\Jm{\bq_h-\bq}}_{L^2(\pa K)}+\tau\norm{\Jm{u_h-u}}_{L^2(\pa K)}\big)\\
&+h_K\Vert\nabla\cdot(\bq-\br_h)\Vert_{L^2(K)}+\Vert\bq-\bq_h\Vert_{L^2(K)}+k^2h_K\Vert v\Vert_{L^2(K)}.\nn
\end{align}
For any $v_h\in V_{h,2}$ and $\br_h\in\bV_h$, we have
\eqn{&\norm{\Jm{\bq_h-\bq}}_{\pa\T_h}+\tau\norm{\Jm{u_h-u}}_{\pa\T_h}\\
&=\norm{\Jm{\bq_h-\br_h+\br_h-\bq}}_{\pa\T_h}+\tau\norm{\Jm{u_h-v_h+v_h-u}}_{\pa\T_h}\\
&\lesssim\norm{\bq_h-\br_h}_{\pa\T_h}+\norm{\br_h-\bq}_{\pa\T_h}+\tau\big(\norm{u_h-v_h}_{\pa\T_h}+\norm{v_h-u}_{\pa\T_h}\big)\\
&\lesssim h^{-\frac12}\big(\norm{\bq_h-\br_h}+\norm{\br_h-\bq}+\tau\norm{u_h-v_h}+\tau\norm{v_h-u}\big)\\
&\quad+h^\frac12\big(\norm{\na_h(\br_h-\bq)}+\tau\norm{\na_h(v_h-u)}\big)\\
&\lesssim h^{-\frac12}\big(\norm{\bq_h-\bq}+\norm{\bq-\br_h}+\tau\norm{u_h-u}+\tau\norm{u-v_h}\big)\\
&\quad+h^\frac12\big(\norm{\na_h(\bq-\br_h)}+\tau\norm{\na_h(u-v_h)}\big).
}
Supposing $kh$ is sufficiently small, plugging the above estimate into the sum of \eqref{ev} over $\T_h$ and using \eqref{eq53}, we obtain
\eq{\label{ev2}
\norm{v}\lesssim~ h\norm{\na_h v}\lesssim& ~h\Vert\bq-\bq_h\Vert+kh\norm{u-u_h}+h\Vert\nabla_h(u-v_h)\Vert+\Vert u-v_h\Vert\\
&+h^2\norm{\na_h(\bq-\br_h)}+h\norm{\bq-\br_h}.\nn}
On the other hand,
\begin{align}
\Vert m_K( \tu-u_h^*)\Vert_{L^2(K)}=\Vert m_K(u-u_h)\Vert_{L^2(K)}&=\Vert m_K(\Pi_1 u-u_h)\Vert_{L^2(K)}\leq\Vert\om_h\Vert_{L^2(K)}.\label{eq56}
\end{align}
Therefore, by using the triangle inequality, \eqref{eq565a}, \eqref{ev2}--\eqref{eq56}, Theorem \ref{thm:erresti}, and \eqref{eom}, we obtain
\eq{\label{erruh*}
\Vert u-u_h^*\Vert\leq&~\Vert u-\tu\Vert+\Vert v\Vert+\bigg(\sum_{K\in\T_h}\Vert m_K(\tu-u_h^*)\Vert_{L^2(K)}^2\bigg)^\frac12\nn\\
\lesssim&~ k^2h\Vert\xi\Vert+kh\norm{\brh}+\inf_{v_h\in V_{h,2}}\big(h\Vert\nabla_h(u-v_h)\Vert+\Vert u-v_h\Vert\big)\\
&+\inf_{\br_h\in\bV_h}\big(h^2\norm{\na_h(\bq-\br_h)}+h\norm{\bq-\br_h}\big)\nn}
which together with Lemma \ref{lem:Pi} and the interpolation error estimates implies that \eqref{errpost} holds. This completes the proof of the theorem.
\end{proof}
\begin{remark}\label{rpost}
{\rm (a)} Comparing \eqref{errpost} with \eqref{eu}, we observe that the postprocessing procedure improves only the interpolation error (from $O(k^2h^2)$ to $O(k^3h^3)$) but leaves the pollution error $O(k^4h^3)$ unchanged.

{\rm (b)} In the next section, we explore how to reduce the pollution by tuning the penalty parameters.
\end{remark}

\section{Dispersion analysis}\label{sec:Dispersion}
In this section, we first define a discrete wave number $k_h$ of HDG method for 1D Helmholtz equation on equidistant grids, and analyze how the penalty parameter $\tau$ affects the phase error $|k_h-k|$. Then, in order to save space, we just list the results for the 2D Helmholtz equation on equilateral triangulations.
 It is well known that the pollution error is of the same order as the phase error for the FEM or CIP-FEM on structured meshes \cite{ihlenburg1997finite,Ainsworth04,Zhu:Burman:Wu:2012}. It will be shown that the phase error of the HDG method may be reduced in both 1D and 2D and even eliminated in 1D by choosing proper penalty parameters, which are expected to be used to reduce the pollution errors in higher dimensions and to eliminate the pollution error in one dimension.

\subsection{One dimensional analysis}
Note that the plane wave $e^{\mathbf{i} k x}$ is a solution to the homogeneous Helmholtz equation 
\eq{\label{Hf0}F(u,k):=-u''-k^2 u=0,}
 that is, the wave number $k$ satisfies the following equation
\eq{\label{dp1}F(e^{\mathbf{i} k x},k)=0.}
The discrete wave number $k_h$ is defined by mimic the above continuous procedure as follows. Let $n>2$ be an integer and $x_i=i h, i=0,1,\cdots, n$ be the nodes of the mesh $\T_h$. Denote by $u_i^\pm:=u_h(x_i\pm0)$, $\bq_i^\pm:=\bq_h(x_i\pm0)$, and $\hat u_i^\pm:=\hat u_h(x_i\pm0)$. By taking $\br_h$ and $v_h$ in \eqref{eq39} and \eqref{eq310} (with $f=0$) to be the nodal basis functions on the $i$-th interval $[x_{i-1}, x_i]\subset (0,1)$, using \eqref{eq311}, and some simple calculations, we obtain the following equations on the $i$-th interval:
\eq{
\frac{h}6\big(2\bq_{i-1}^++\bq_i^-\big)+\frac12\big(u_{i-1}^++u_i^-\big)&=\hat u_{i-1},\\
\frac{h}6\big(\bq_{i-1}^++2\bq_i^-\big)-\frac12\big(u_{i-1}^++u_i^-\big)&=-\hat u_i,\\
-\frac12\big(\bq_{i-1}^++\bq_i^-\big)+\frac{k^2h}6\big(2u_{i-1}^++u_i^-\big)&=-\bq_{i-1}^++\tau\big(u_{i-1}^+-\hat u_{i-1}\big),\\
\frac12\big(\bq_{i-1}^++\bq_i^-\big)+\frac{k^2h}6\big(u_{i-1}^++2u_i^-\big)&=\bq_i^-+\tau\big(u_i^--\hat u_i\big).
} 
By solving the above system of four equations, $u_{i-1}^+,u_i^-,\bq_{i-1}^+$, and $\bq_i^-$ are explicitly expressed in  $\hat u_{i-1}$ and $\hat u_i$ (we omitted the specific expressions to save space), and hence  $u_i^+$ and $\bq_i^+$ are expressed in $\hat u_i$ and $\hat u_{i+1}$. Then we substitute these expressions into \eqref{eq311} and \eqref{eq312} with $e=x_i$ and obtain the HDG equation of $\hat u_h$ at the interior node $x_i$:
\eq{\label{HDGhuh}F_h(\hat u_h,k):=S\hat u_{i-1}+R\hat u_i+S\hat u_{i+1}=0,}
where
\eq{
S:=&-1-\frac{k^2h^2(12\tau h-3k^2h^2+2\tau^2h^2)}{(12+2\tau h-k^2h^2)(6\tau h-k^2h^2)},\label{S}\\
R:=&2-\frac{2k^2h^2(24\tau h-3k^2h^2+4\tau^2h^2-k^2h^2\tau h )}{(12+2\tau h-k^2h^2)(6\tau h-k^2h^2)}.\label{R}}
Similar to \eqref{dp1}, the discrete wave number $k_h$ is defined as the solution to following equation
\eq{\label{dwn}
F_h(I_he^{\mathbf{i} k x},k_h)=0,}
which is located at the solution branch near to the wave number $k$. Clearly, the above nonlinear equation has multiple solutions, while the other solutions are called spurious wave numbers and will not discussed in this paper. 
For simplicity, denote by $t=kh$, $t_h=k_h h$, and $s=\tau h$. Using \eqref{HDGhuh}--\eqref{R}, \eqref{dwn} may be rewritten as
\eqn{2(12+2s-t_h^2)(6s-t_h^2)(1-\cos t)-2t_h^2\big(24s-3t_h^2+4s^2-t_h^2s+(12s-3t_h^2+2s^2)\cos t\big)=0.}
The difference $\de:=t_h^2-t^2$ between $t_h^2$ and $t^2$ satisfies the following equation:
\eq{\label{kh}
a_2\de^2+a_1\de&+a_0=0,\\
a_2:=&4+s+2\cos t,\notag\\
a_1:=&(12-4s-2s^2+4t^2)(\cos t-1)-2(6+s)(3s-t^2),\notag\\
a_0:=&(2b(t^2+6)+t^4)(s^2-t^2) + (4b(t^2+18)+3t^4)s,\quad b:=1-t^2/2-\cos t.\notag
}
Therefore, the discrete wave number is given explicitly by
\eq{\label{th2}
k_h=\frac{t_h}h,\quad t_h^2=t^2+\frac{-a_1-\sqrt{a_1^2-4a_0a_2}}{2a_2}.}
Clearly, $k_h=k$ (i.e $\de=0$) if $\tau$ (i.e. $s$) is so chosen such that $a_0=0$. That is, we have the following lemma which says that the phase error of the HDG solution for the 1D problem may be eliminated completely by choosing proper penalty parameter $\tau$.
\begin{lemma}\label{lem:pf} $k_h=k$ if $kh\le \pi$ and
\eq{\label{tau_pf}
\tau=\tau_o:=k\frac{\sqrt{(4b(t^2+18)+3t^4)^2+(4b(t^3+6t)+2t^5)^2}-(4b(t^2+18)+3t^4)}{(4b(t^3+6t)+2t^5)}}
where $b=1-t^2/2-\cos t$.
\end{lemma}
\begin{proof} It is easy to check that, for $0<t\le\pi$,
\eqn{2b(t^2+6)+t^4>t^4\big(2(-1/4!+t^2/6!-t^4/8!)(t^2+6)+1\big)>0.} 
Then the proof follows by solving $a_0=0$ for $s$.
\end{proof}
\begin{remark}\label{rpf}
{\rm (a)} Numerical tests in the next section show that the HDG method with $\tau=\tau_o$ is pollution-free for 1D problems. The rigorous analysis will be considered in a future work. 

{\rm (b)} Such a parameter may also be applied to reduce the pollution errors of the linear HDG method for Helmholtz problems in higher dimensions on Cartesian meshes. Based on the work \cite{Babuvska:lvoM:Sauter:2000}, the pollution error of the linear HDG method in higher dimensions can not be eliminated completely by tuning the penalty parameter, since there are infinitely many directions of wave propagations. Clearly, for problems higher dimensions, it is unnecessary to use the exact (but complicated) $\tau_o$, whose proper approximations  be sufficient. By some simple calculations, we have
\eqn{\tau_o=k\big(1+t/15+O(t^2)\big).} 
\end{remark}

The following lemma gives leading orders of the phase errors of the linear HDG method with $\tau=k, k(1+\frac{t}{15}), \mathbf{i} k$, and $\frac{\mathbf{i}}{h}$ (used in \cite{Chen:Lu:Xu:2012}), respectively.
\begin{lemma}\label{lem:phaseerror} We have
\eqn{k_h-k=\begin{cases}
\big(\frac{1}{888} + \frac{\mathbf{i}}{148}\big)k^3h^2(1+O(t^2))&\text{if } \tau=\mathbf{i}\frac{1}{h};\\
\frac{\mathbf{i}}{72}k^4h^3(1+O(t))&\text{if } \tau=\mathbf{i}k;\\
-\frac{1}{1080}k^5h^4(1+O(t^2))&\text{if } \tau=k;\\
-\frac{1}{32400}k^6h^5(1+O(t))&\text{if } \tau=k(1+\frac{t}{15}).
\end{cases}}
\end{lemma}
\begin{proof} We only describe the ideas of the proof but omit the technical details.  Noting that $\de^2$ doubles the order of the infinitesimal $\de$, the solution to $a_1\de+a_0=0$ is a good approximation to that of \eqref{kh}, and hence from \eqref{th2}, $t_h^2-t^2\approx -a_0/a_1$. Then the results are obtained by substituting $s=\mathbf{i}, \mathbf{i}t, t, $ and $t+\frac{t^2}{15}$ into $-a_0/a_1$ and doing Taylor expansions, respectively. 
\end{proof}
\begin{remark}\label{rpe}
{\rm (a)} It is easy to show that $|k_h-k|=O(k^3h^2)$ if $\tau=O(1/h)$, which is of the same order as the pollution error proved in \cite{Chen:Lu:Xu:2012}.

{\rm (b)} It is also holds that $|k_h-k|=O(k^4h^3)$ if $\tau\eqsim k$, which matches the pollution errors in the estimates given in Theorem~\ref{thm:erresti}.

{\rm (c)} If we set $\bq=\nabla u$ instead of $\bq=-\nabla u$ in \eqref{eq33} to formulate the HDG method, the conditions for the 3rd and 4th estimates should be replaced by $\tau=-k$ and $\tau=k(-1+\frac{t}{15})$, respectively.
\end{remark}

\subsection{Two dimensional analysis on equilateral triangulations} By following the above procedure for 1D case but with a little more complicated calculations, we may define the discrete wave number $k_h$ of the HDG method on equilateral triangulations and obtain the following lemma which gives leading orders of the phase errors of the linear HDG method with $\tau=k, \mathbf{i} k, \frac{\sqrt{2}}{2}k\big(1+\frac{\sqrt{3}}{64}kh\big)$, and $\frac{\mathbf{i}}{h}$ (used in \cite{Chen:Lu:Xu:2012}), respectively. The details are omitted.
\begin{lemma}\label{lem:phaseerror2d} We have
\eqn{|k_h-k|\begin{cases}
=\frac{\sqrt{149961}}{129696}k^3h^2(1+O(t^2))&\text{if } \tau=\mathbf{i}\frac{1}{h};\\
=\frac{\sqrt{3}}{384}k^4h^3(1+O(t))&\text{if } \tau=\mathbf{i}k;\\
=\frac{\sqrt{3}}{1152}k^4h^3(1+O(t))&\text{if } \tau=k;\\
\le \frac{1}{46080}k^5h^4(1+O(t))&\text{if } \tau=\frac{\sqrt{2}}{2}k\big(1+\frac{\sqrt{3}}{64}kh\big).
\end{cases}}
\end{lemma}

Note that the phase errors for $\tau=k, \mathbf{i} k$ and $\frac{\mathbf{i}}{h}$ coincide with the pollution errors $O(k^4h^3)$ in  
Theorem~\ref{thm:erresti} and $O(k^3h^2)$ in \cite[Theorem 6.1]{Chen:Lu:Xu:2012}, respectively. We remark that, for any constant $a$, choosing $\tau=\frac{\sqrt{2}}{2}k\big(1+a kh\big)$ may improves the phase error to $O(k^5h^4)$, and the optimal value of $a$ is $\frac{\sqrt{3}}{64}$. Such a $\tau$ is also expected to reduce further the pollution error, which will be verified numerically in the next section.

\section{Numerical examples}
In this section, we present two numerical examples to verify our error estimates for the linear HDG method and examine the influence of the penalty parameter $\tau$ on the pollution errors. Denote the relative $L^2$ errors of $u_h$, $\bq_h$, and theirs interpolations by 
\eqn{e_u:=\frac{\Vert u-u_h\Vert}{\Vert u\Vert},\;e_\bq:=\frac{\Vert \bq-\bq_h\Vert}{\Vert \bq\Vert}, \;e_u^I:=\frac{\Vert u-I_hu\Vert}{\Vert u\Vert},\;\text{ and } e_\bq^I:=\frac{\Vert \bq-I_h\bq\Vert}{\Vert \bq\Vert}.}

\begin{example}\label{ex1} An 1D Helmholtz problem \eqref{eq31}--\eqref{eq32} with $\Om=(0,1)$, $f=0$, and $g(0)=g(1)=1$. The exact solution is given by
\eqn{u=\frac{1}{2\mathbf{i}k}\big(e^{-\mathbf{i} kx}+e^{\mathbf{i} k(x-1)}\big).}
\end{example}

First we test the relative $L^2$ errors of $u_h$ and $\bq_h$ for fixed wave number $k=10, 100,$ and $200$, respectively, and $h=1/2,1/3,\cdots, 1/10000$. Figure~\ref{F1} plots results for $\tau=\mathbf{i}/h$. All the interpolation errors decay at the full rate of $O(h^2)$ for $h$ less than around half wavelength and are pollution-free. The relative $L^2$ error of $u_h$ is almost the same as the  that of the interpolant $I_hu$ for the small wave number $k=10$. For larger wave number $k=100$, $e_u$ first stays around $100\%$ and then decays at the rate of $O(h^2)$ after a point later than that of $e_u^I$. For $k=200$ the gap between the decaying points is even larger. Similar results hold for $\bq_h$ (see Figure~\ref{F1}(right)) except the asymptotic convergence rates of $e_\bq$ are $O(h)$ which is not full. This figure clearly shows the existence of pollution effect for large wave number and verifies the following estimates given in \cite{Chen:Lu:Xu:2012} for the linear HDG method with $\tau=\mathbf{i}/h$:
\eq{\label{cxlest}
k\norm{u-u_h}_{L^2(\Om)}=O(k^2h^2+k^3h^2),\quad \norm{\bq-\bq_h}_{L^2(\Om)}=O(kh+k^3h^2).}
\begin{figure}[htbp]
\begin{center}
\includegraphics[scale=0.5]{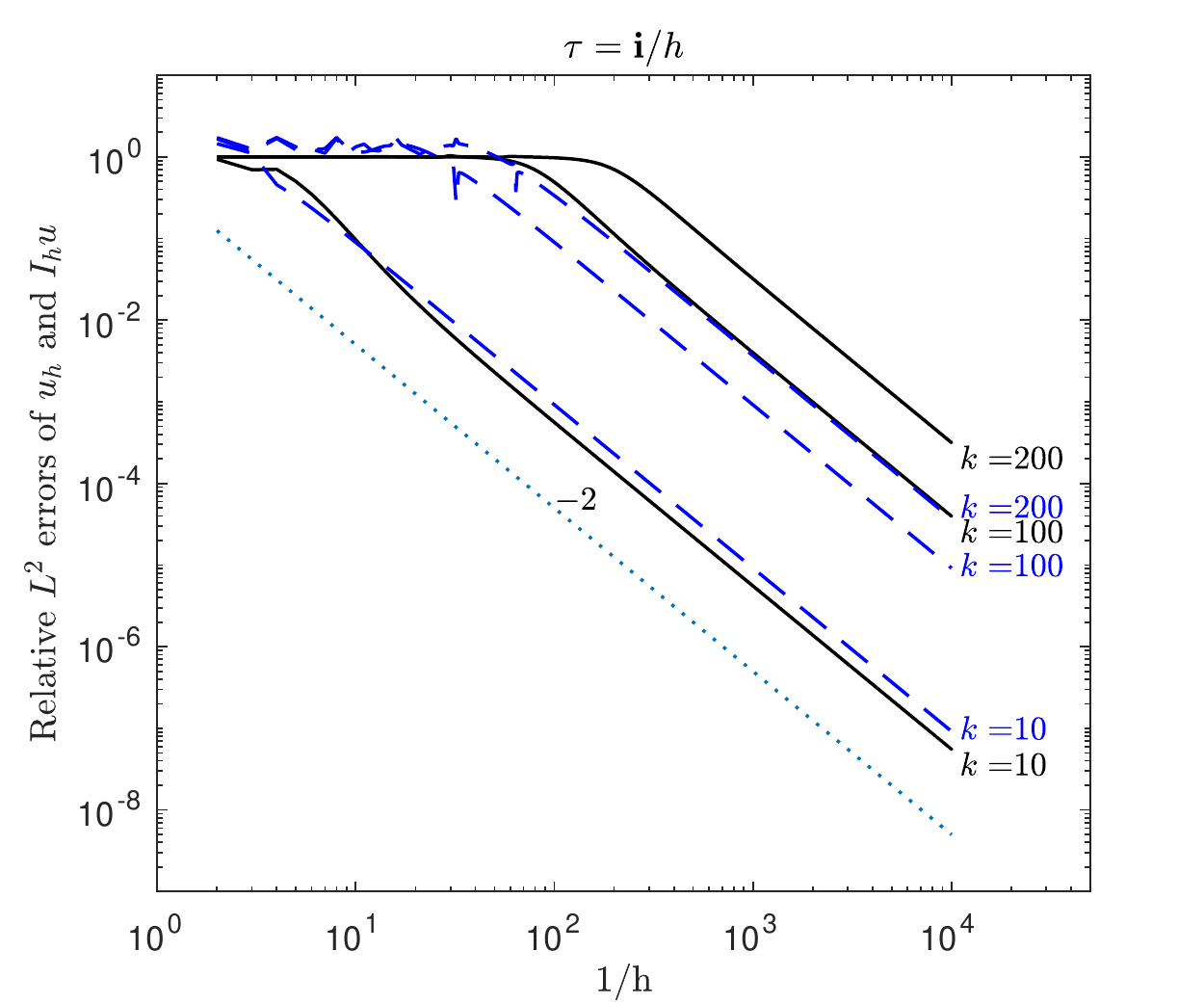}
\includegraphics[scale=0.5]{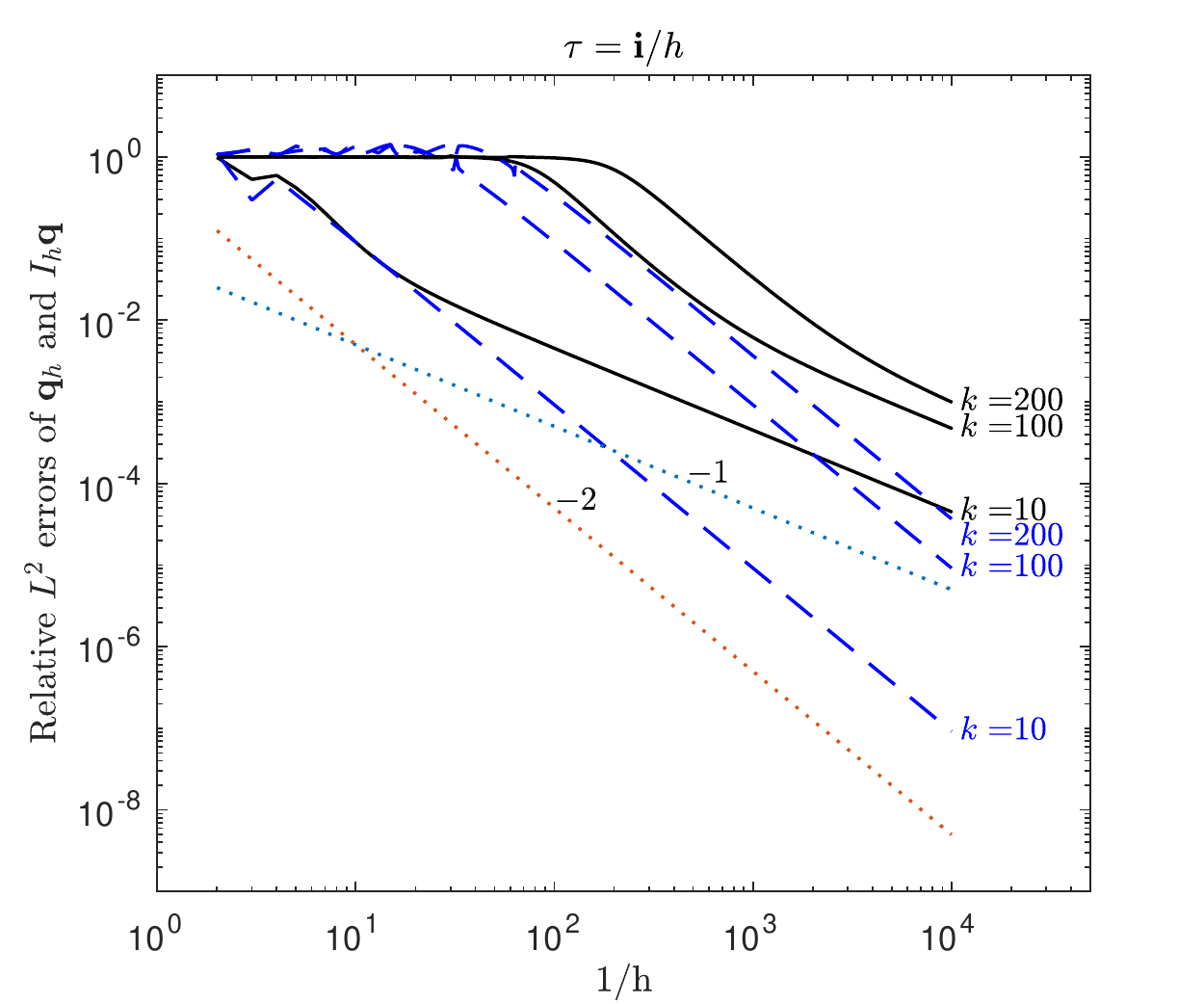}
\end{center}
\caption{Example  \ref{ex1}: $\tau=\mathbf{i}/h$. 
 Relative $L^2$ errors of $u_h$, $\bq_h$ (solid) and theirs interpolations (dashed) versus $1/h$. Dotted lines gives reference slopes. \label{F1}}
\end{figure}
We remark that the linear HDG method with $\tau=\mathbf{i}/h$ behaves much like the linear FEM (see Figure~\ref{F2}). 
\begin{figure}[htbp]
\begin{center}
\includegraphics[scale=0.5]{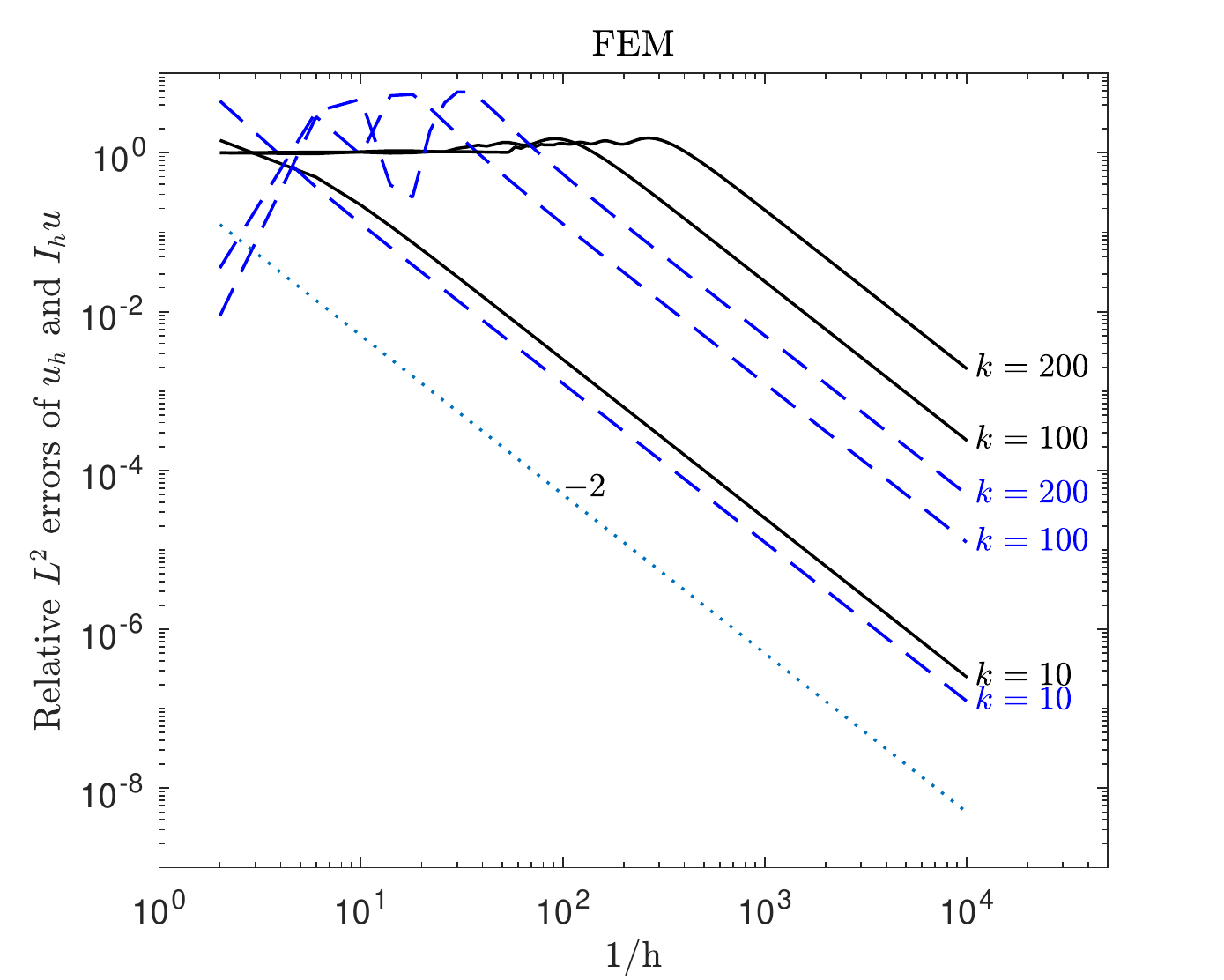}
\includegraphics[scale=0.5]{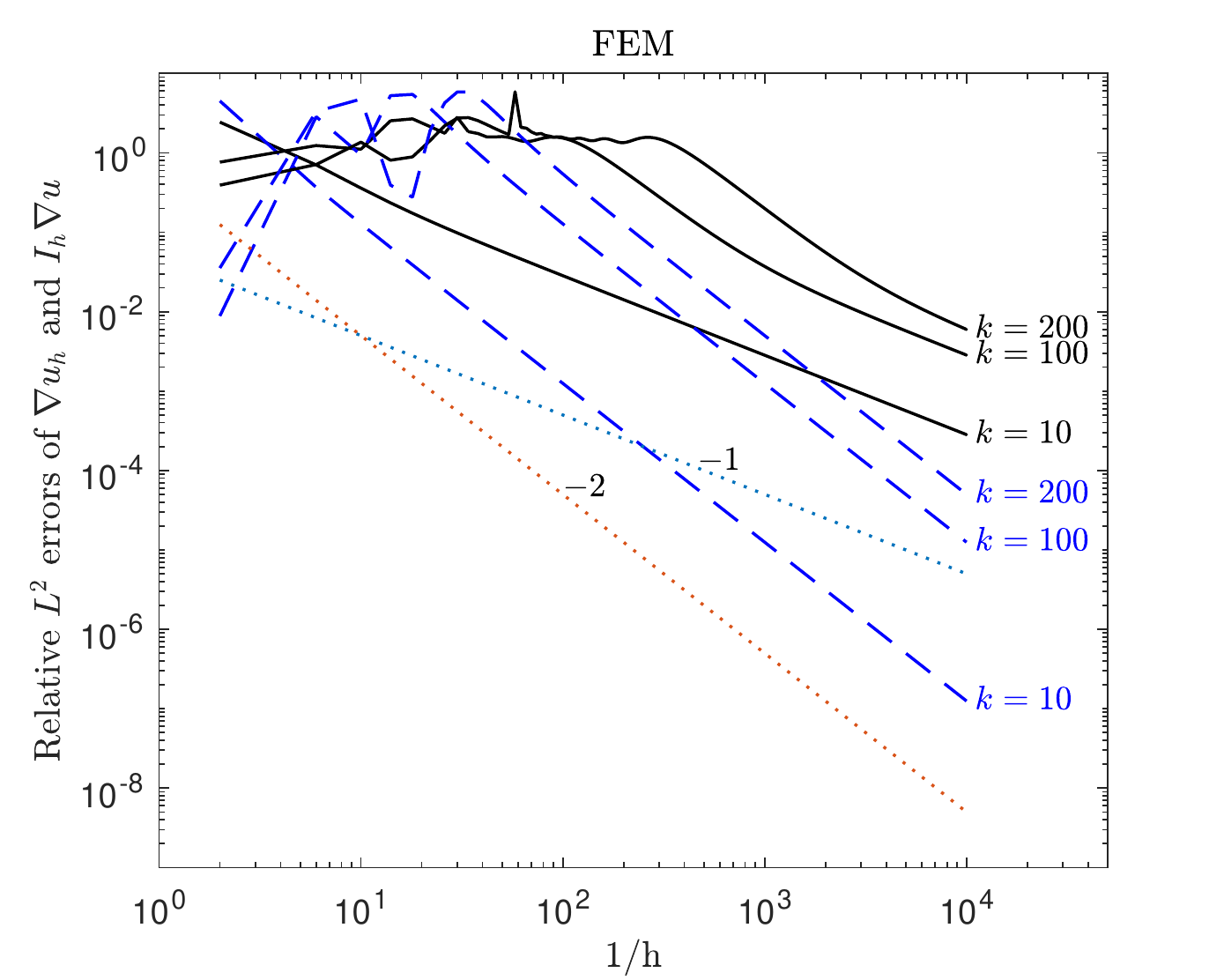}
\end{center}
\caption{Example  \ref{ex1}: FEM. 
 Relative $L^2$ errors of $u_h$, $\na u_h$ (solid) and $I_hu, \na I_hu$ (dashed) versus $1/h$. Dotted lines gives reference slopes. \label{F2}}
\end{figure}
Figure~\ref{F3} plots results for $\tau=\mathbf{i}k$, which clearly performs better than the case of $\tau=\mathbf{i}/h$. On the one hand, $e_\bq$ converges asymptotically at the full rate of $O(h^2)$ in $h$. On the other hand, both $e_u$ and $e_\bq$ decay at a rate faster than $O(h^2)$ after the decaying points and then approaches to $e_u^I$ and $e_\bq^I$, respectively, which indicates that the order of the pollution errors in $h$ is higher than that of interpolation errors, as the error estimates \eqref{eu}--\eqref{eq} show.
\begin{figure}[htbp]
\begin{center}
\includegraphics[scale=0.5]{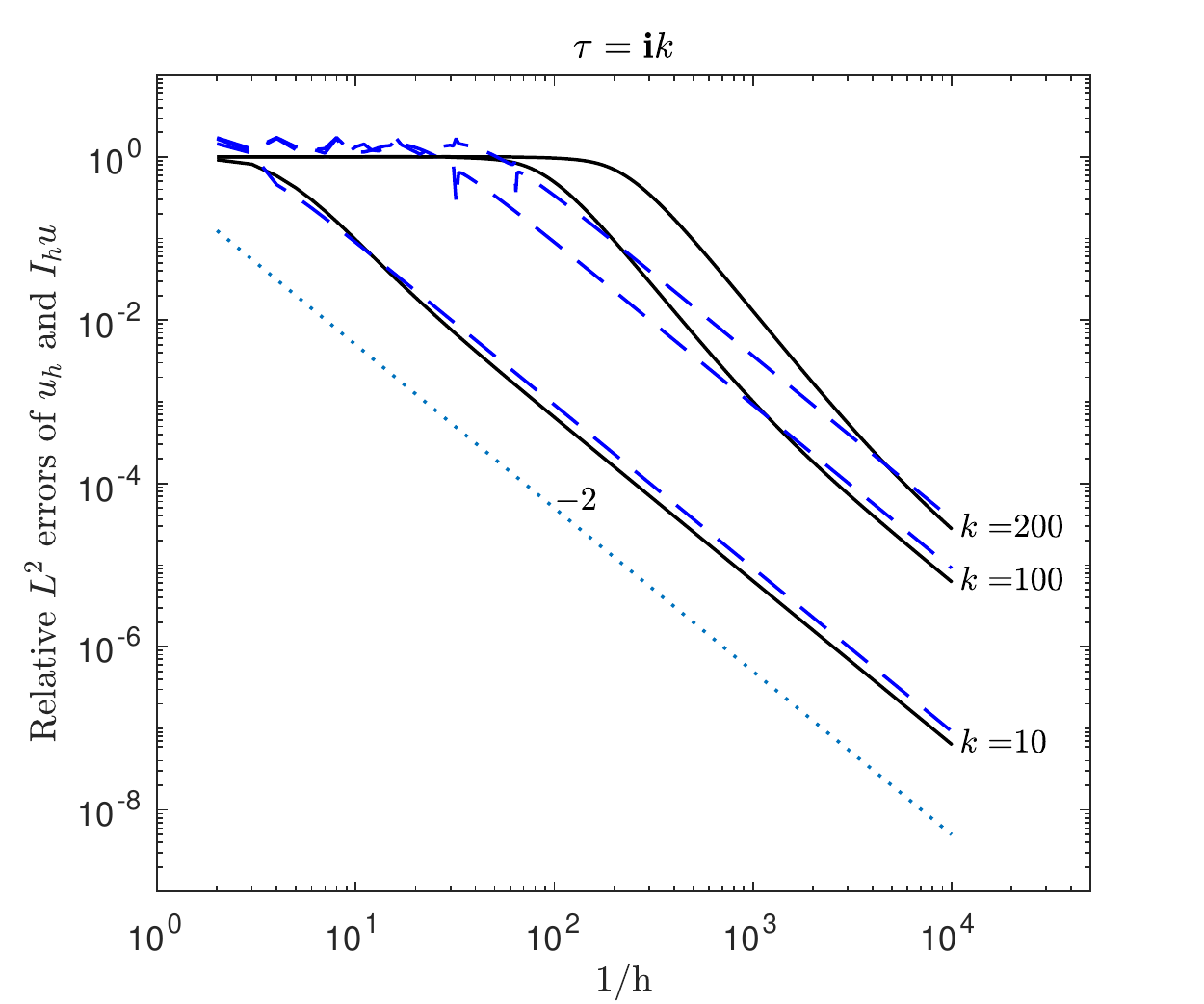}
\includegraphics[scale=0.5]{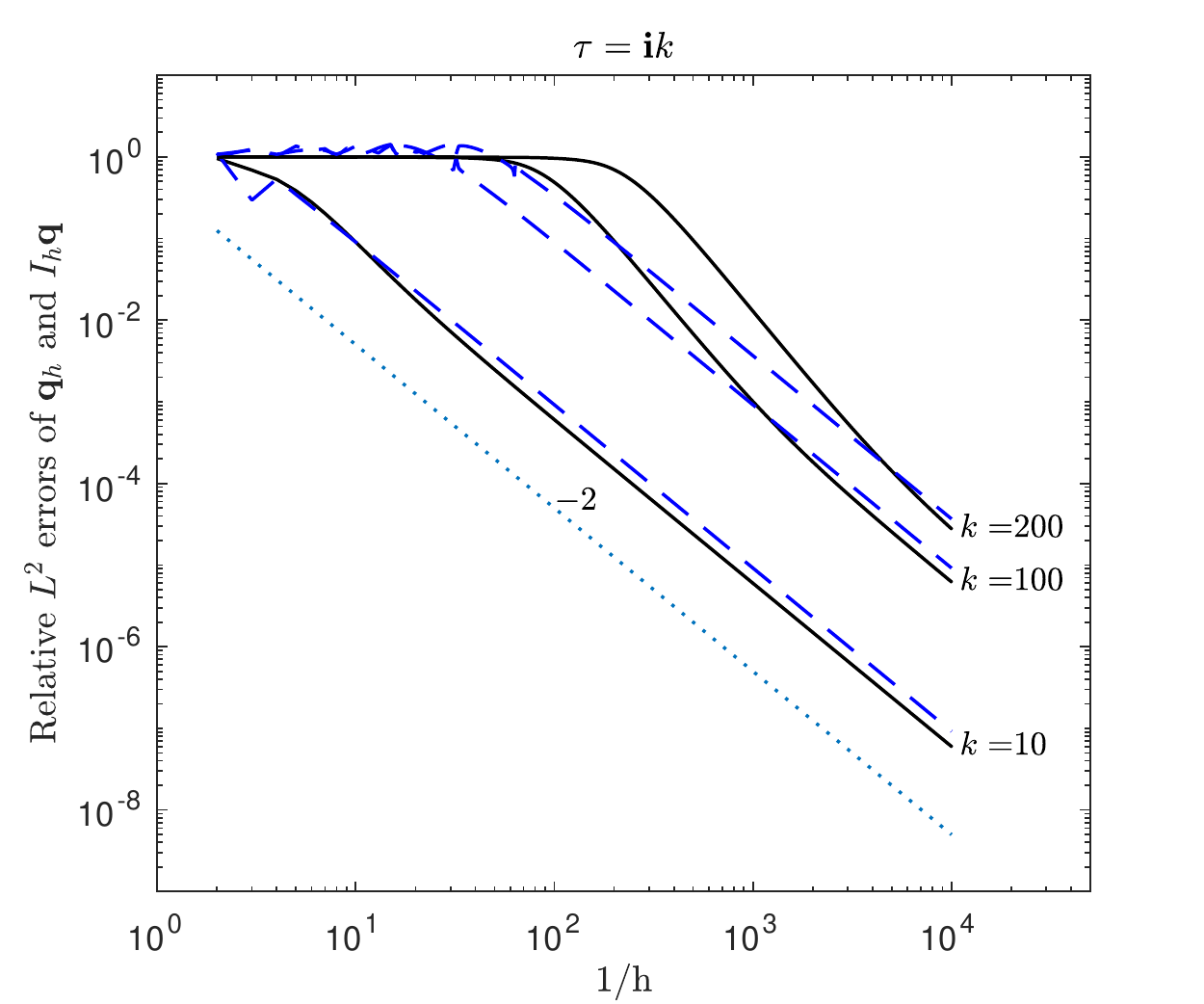}
\end{center}
\caption{Example  \ref{ex1}: $\tau=\mathbf{i} k$. 
 Relative $L^2$ errors of $u_h$, $\bq_h$ (solid) and theirs interpolations (dashed) versus $1/h$. Dotted lines gives reference slopes. \label{F3}}
\end{figure}
Figure~\ref{F4} shows that the pollution errors of the HDG methods are reduced significantly by setting $\tau=k$ and even further by setting $\tau=k(1+\frac{kh}{15})$ (see Figure~\ref{F5}), which verifies Lemma~\ref{lem:phaseerror}.
\begin{figure}[htbp]
\begin{center}
\includegraphics[scale=0.5]{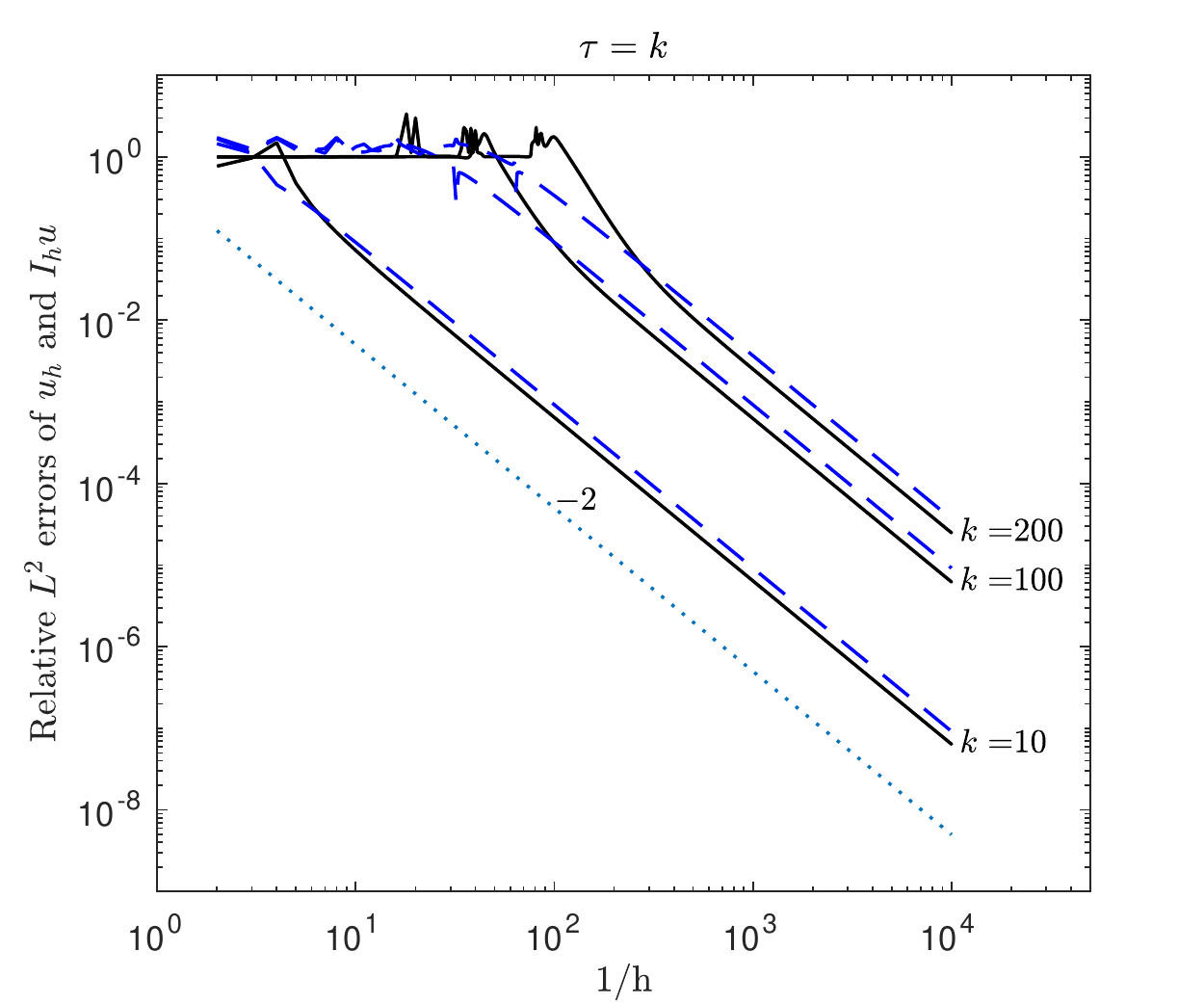}
\includegraphics[scale=0.5]{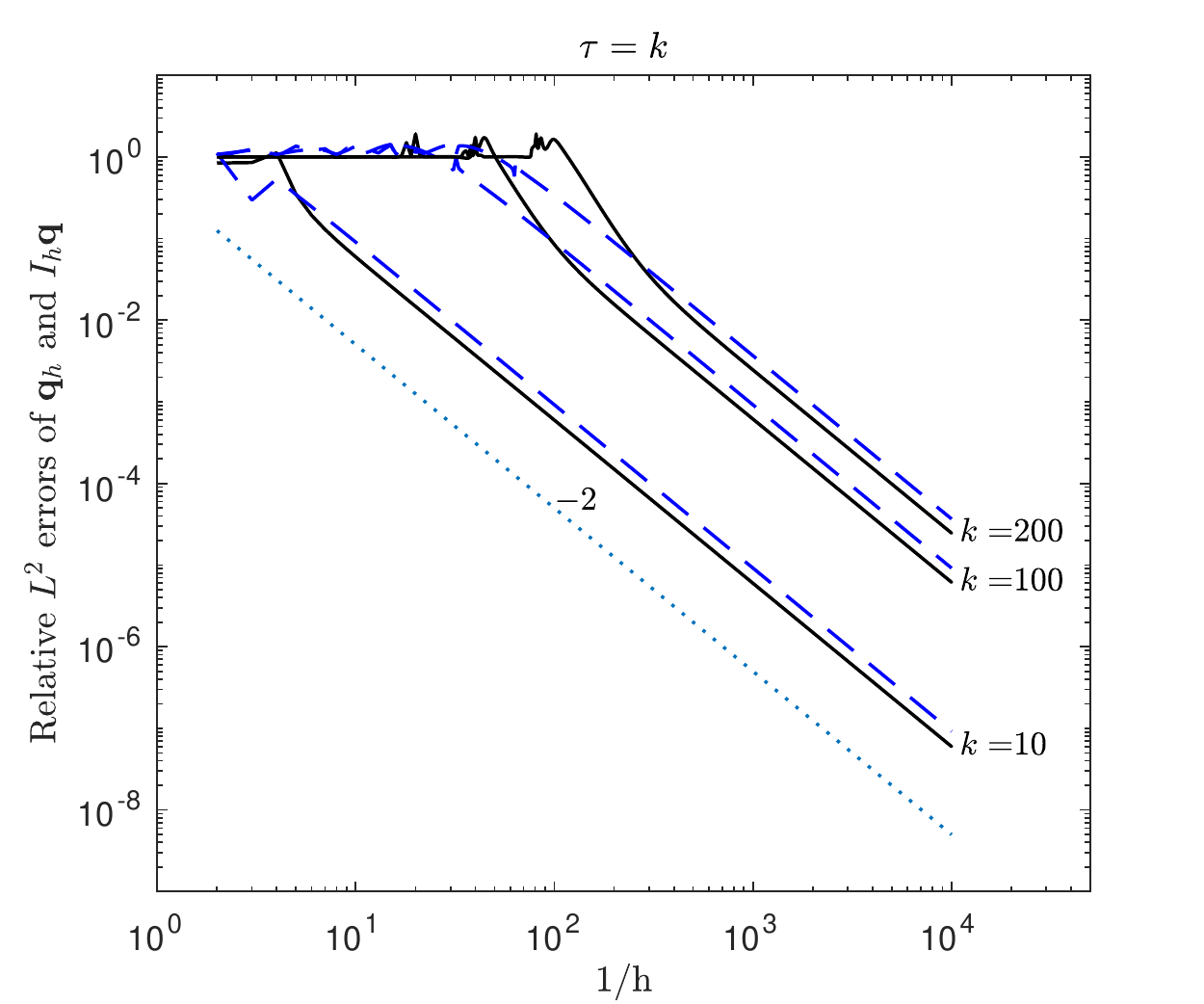}
\end{center}
\caption{Example  \ref{ex1}: $\tau=k$. 
 Relative $L^2$ errors of $u_h$, $\bq_h$ (solid) and theirs interpolations (dashed) versus $1/h$. Dotted lines gives reference slopes. \label{F4}}
\end{figure}
\begin{figure}[htbp]
\begin{center}
\includegraphics[scale=0.5]{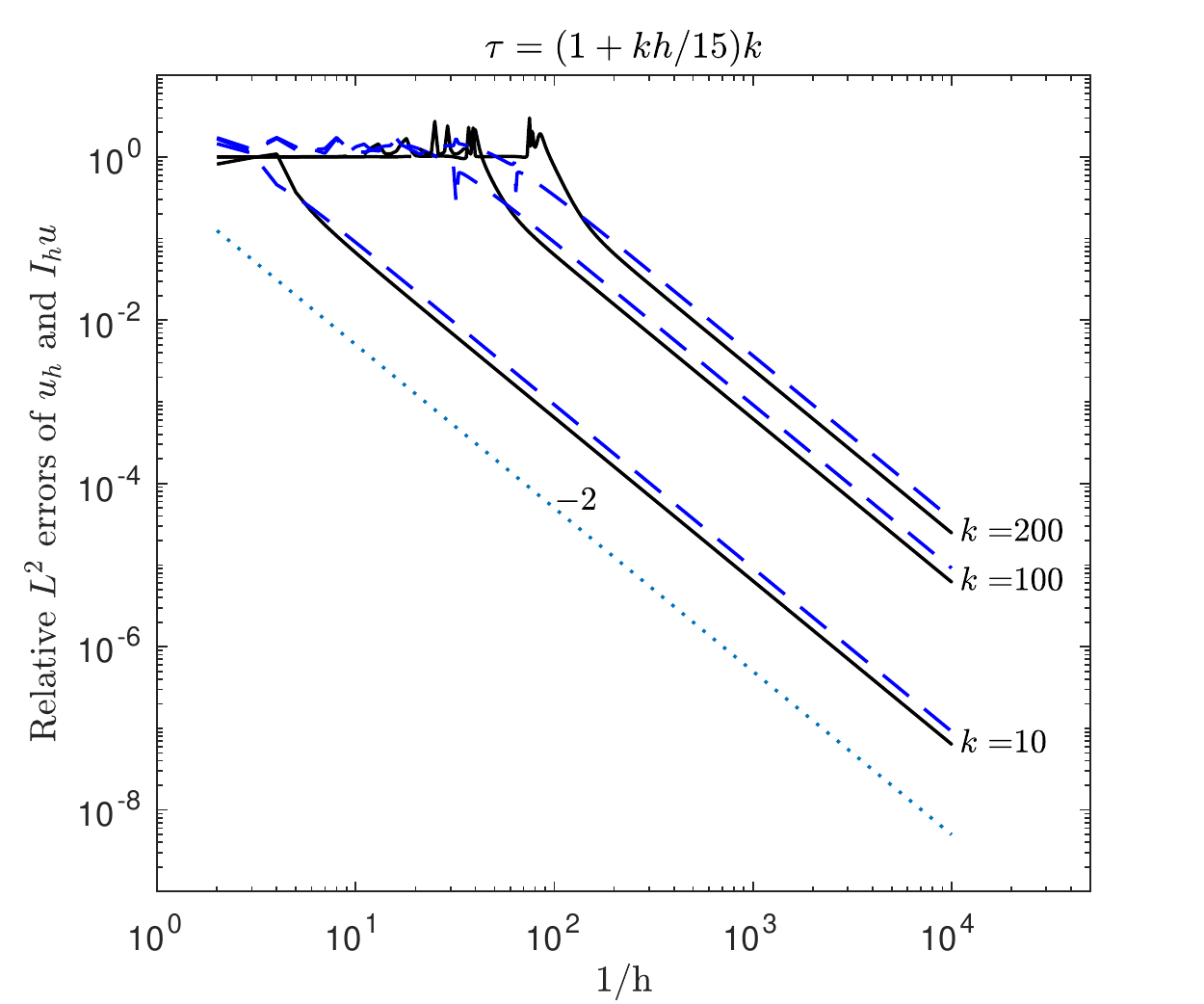}
\includegraphics[scale=0.5]{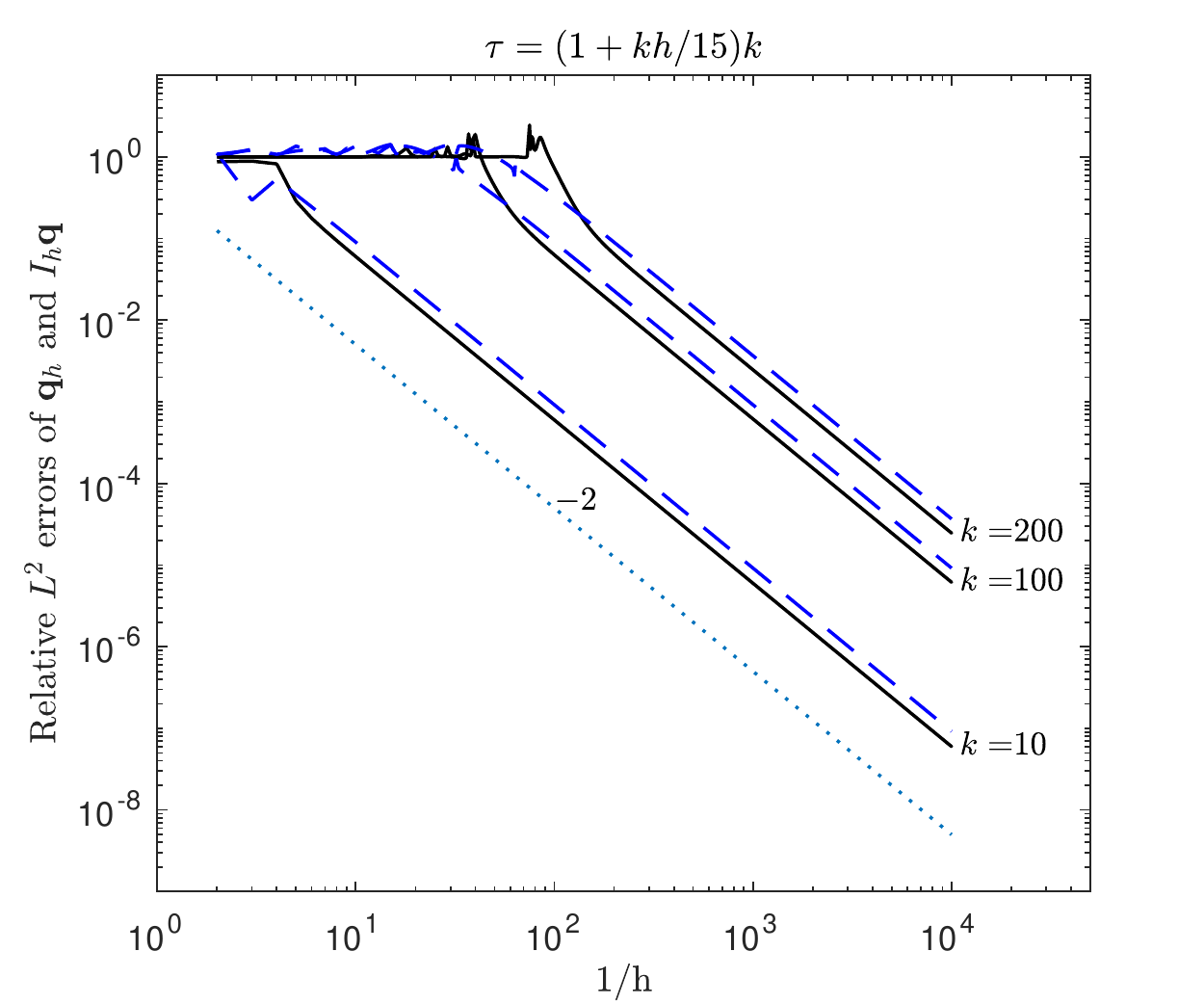}
\end{center}
\caption{Example  \ref{ex1}: $\tau=k(1+\frac{kh}{15})$. 
 Relative $L^2$ errors of $u_h$, $\bq_h$ (solid) and theirs interpolations (dashed) versus $1/h$. Dotted lines gives reference slopes.\label{F5}}
\end{figure}

Next we illustrate the pollution effects by fixing $kh=1$ and letting $k$ varies from 1 to 500. Figure~\ref{F6} plots the relative $L^2$ errors of $u_h$ (left) and $\bq_h$ (right) for $\tau=\mathbf{i}k, k$, and $k(1+\frac{kh}{15})$, respectively. The relative interpolation errors $e_u^I$ and $e_\bq^I$ (dashed lines) keep almost unchanged for $k=1,2,\cdots,500$, which agrees theirs theoretical estimates  $O(k^2h^2)$ and are pollution-free. The HDG solutions for $\tau=\mathbf{i}k$ began to show obvious pollution effect for $k$ larger than about $10$. The pollution effect is reduced significantly for $\tau=k$ and almost disappears for $\tau=k(1+\frac{kh}{15})$ and $k$ up to 500. 
\begin{figure}[htbp]
\begin{center}
\includegraphics[scale=0.5]{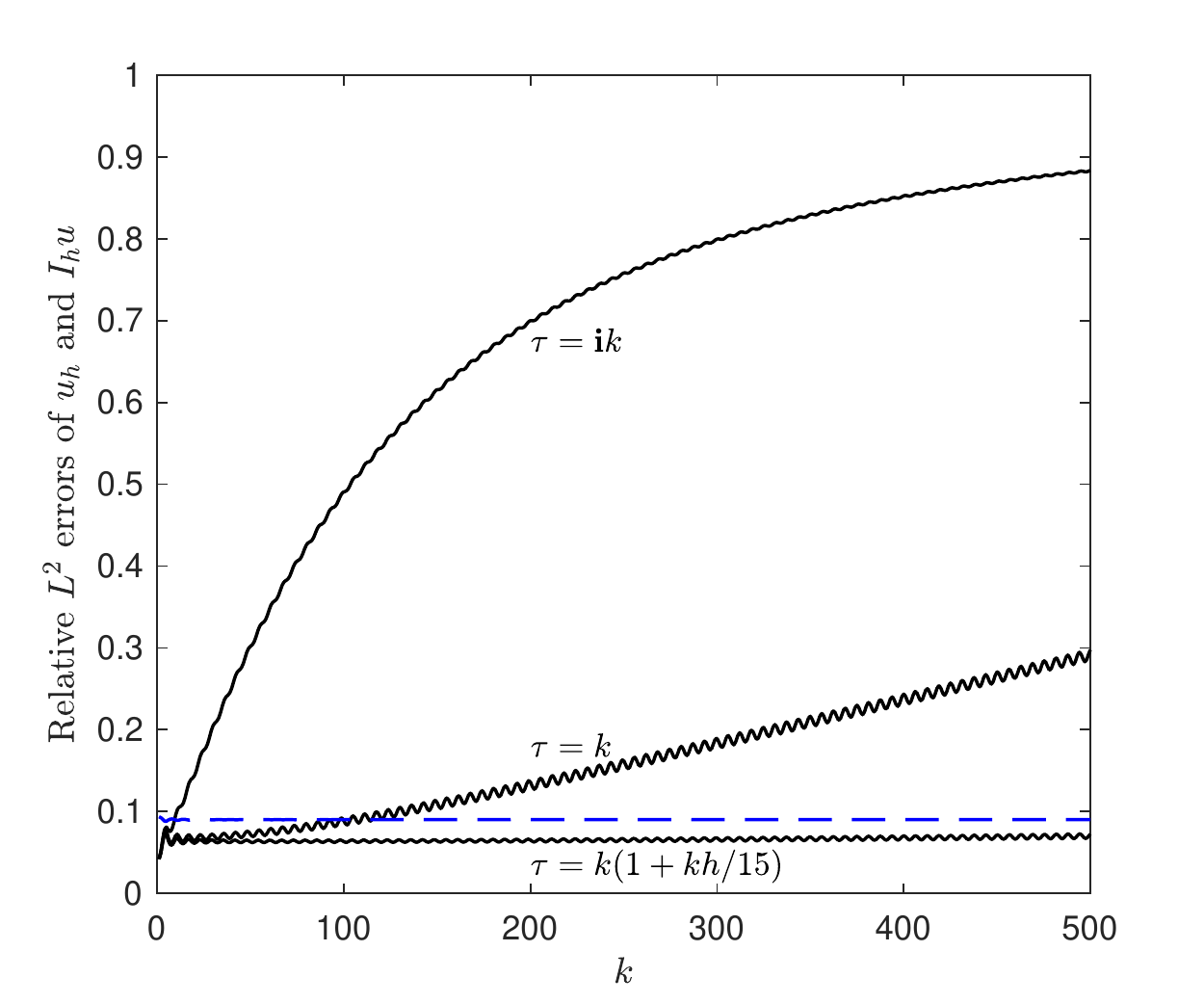}
\includegraphics[scale=0.5]{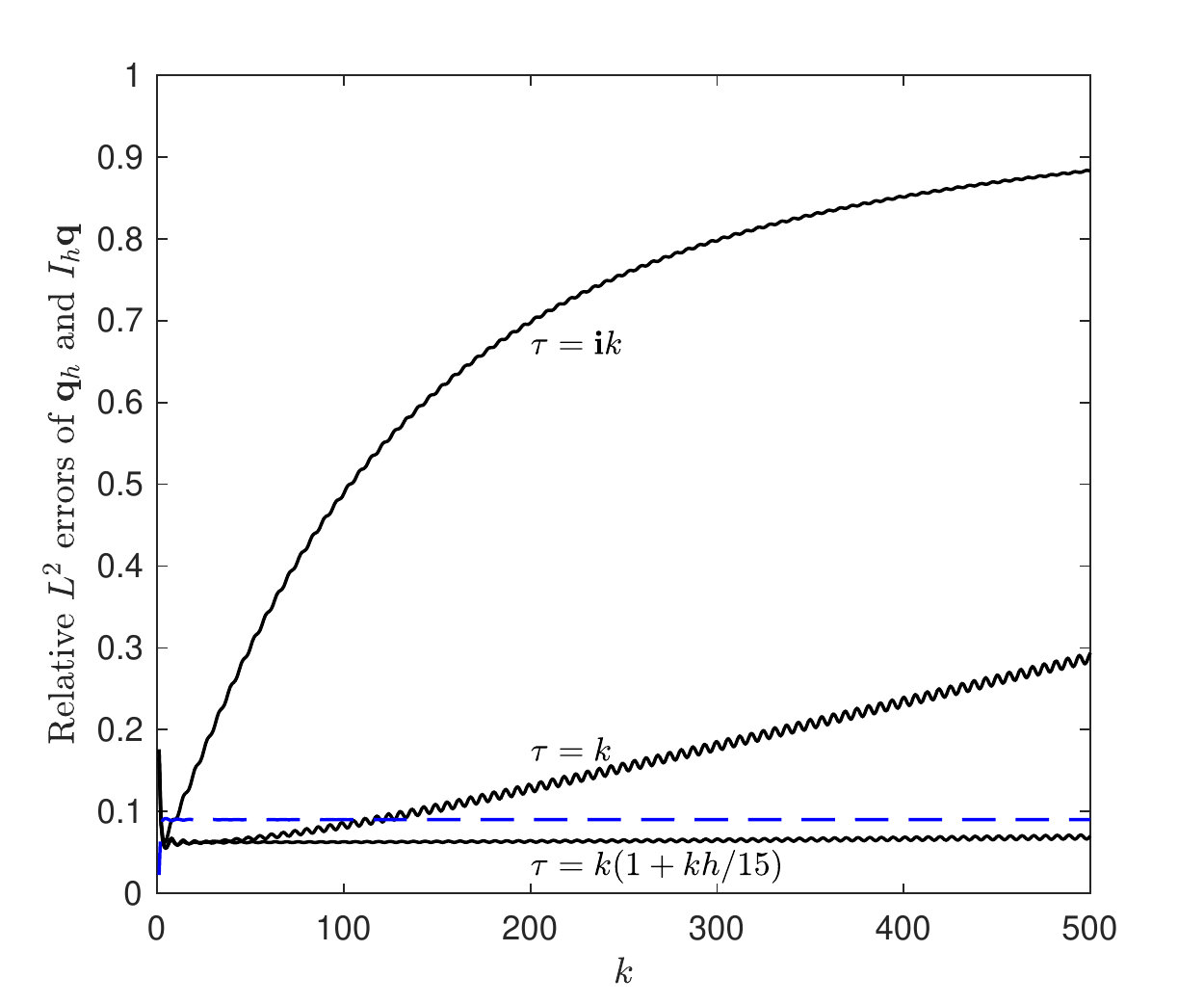}
\end{center}
\caption{Example  \ref{ex1}: $kh=1$.  
 Relative $L^2$ errors $e_u$ (left) and $e_\bq$ (right) versus $k=1,2,\cdots,500$ for $\tau=\mathbf{i}k, k$, and $k(1+\frac{kh}{15})$, respectively. The dashed lines plot the relative interpolation errors.\label{F6}}
\end{figure}

{In the following}, we verify more precisely the pollution terms in the error bounds of $u_h$ and $\bq_h$. To do so, we introduce the definition of the critical mesh sizes with respect to a given tolerance (cf. \cite[Definition 7.1]{WuHaijun:2014}.

\begin{definition}
 Given a tolerance $\ep$ and a wave number $k$, the critical mesh size $h=h(k, \ep)$ with respect to $u_h$ ( or $\bq_h$) is defined by the maximum mesh size such that the relative $L^2$ errors of $u_h$ ( or $\bq_h$) is less than or equal to $\ep$.
\end{definition}
It is clear that if the pollution term is of order $k^{m+1}h^m$ for an integer $m>1$ (see e.g. \eqref{eu},\eqref{eq}, or \eqref{cxlest}), then $h(k,\ep)$ should be proportional to $k^{-\frac{m+1}{m}}$ for $k$ large enough. Figure~\ref{F7} plots the critical mesh sizes $h(k,0.1)$ with respect to $u_h$ (left) and $\bq_h$ (right) with $\tau=\mathbf{i}/h, \mathbf{i}k, k$, and $k(1+\frac{kh}{15})$, respectively. 
It is shown that, for both $u_h$ and $\bq_h$, $h(k,0.1)=O(k^{-\frac32})$ if $\tau=\mathbf{i}/h$, $h(k,0.1)=O(k^{-\frac43})$ if $\tau=\mathbf{i}k$, $h(k,0.1)=O(k^{-\frac54})$ if $\tau=k$, and  $h(k,0.1)=O(k^{-\frac65})$ if $\tau=k(1+\frac{kh}{15})$. The first two observations verify the pollution terms in the error estimates \eqref{cxlest} and \eqref{eu}--\eqref{eq}, respectively, while the last two observations indicate that the pollution terms  should be $O(k^5h^4)$ for $\tau=k$ and $O(k^6h^5)$ for $\tau=k(1+\frac{kh}{15})$, which coincide the corresponding phase errors in Lemma~\ref{lem:phaseerror}, respectively.
\begin{figure}[htbp]
\begin{center}
\includegraphics[scale=0.5]{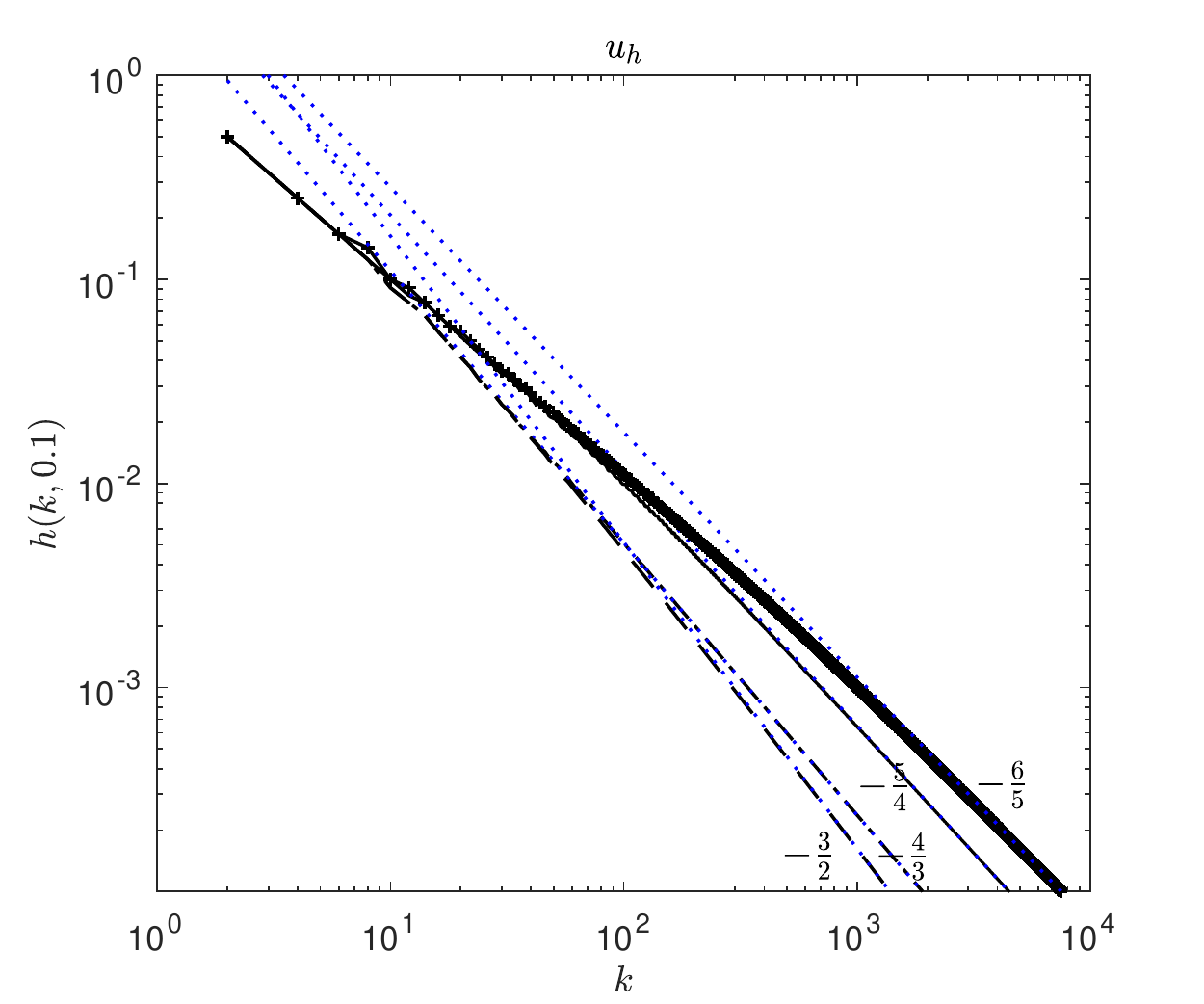}
\includegraphics[scale=0.5]{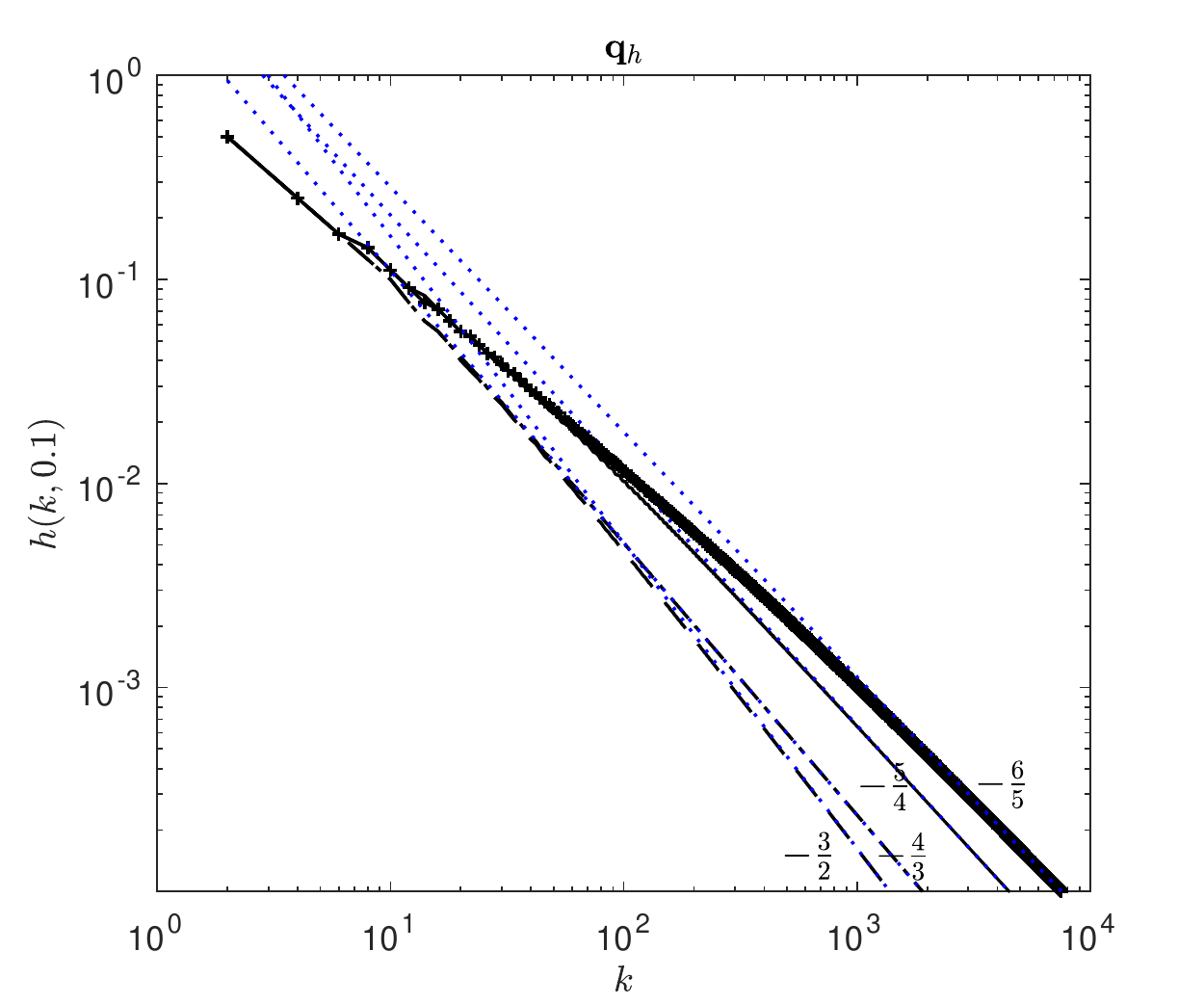}
\end{center}
\caption{Example  \ref{ex1}: Critical mesh sizes with respect to $u_h$ (left) and $\bq_h$ (right) with $\tau=\mathbf{i}/h$ (dashed), $\mathbf{i}k$ (dash dotted), $k$ (solid), and $k(1+\frac{kh}{15})$ (solid marked by $+$), respectively. The dotted lines give reference slopes. \label{F7}}
\end{figure}

\begin{example}\label{ex2} An 2D Helmholtz problem \eqref{eq31}--\eqref{eq32} with $\Om$ to be the unit hexagon centered at the origin, $f$ and $g$ is also chosen such the exact solution is given by
\eqn{
u=J_0(kr),
}
in polar coordinates, where $J_0(z)$ is Bessel function of the first kind.
\end{example}

The domain $\Om$ is triangulated into equilateral triangles of equal size. 
Figure~\ref{F8} shows the relative $L^2$ errors of the HDG solution $u_h$ and the FE solution (left),  $\bq_h$ and the gradient of the FE solution (right), and the interpolations for, $k=1,2,\cdots,500$ and $\tau=\mathbf{i}k, k$, and $\frac{\sqrt{2}}{2}k\big(1+\frac{\sqrt{3}}{64}kh\big)$, respectively. It is shown that the pollution effect of the HDG method is weaker that the FEM and is almost invisible for $\tau=\frac{\sqrt{2}}{2}k\big(1+\frac{\sqrt{3}}{64}kh\big)$, which verifies our theoretical findings in Theorem~\ref{thm:erresti} and Lemma~\ref{lem:phaseerror2d}. 
\begin{figure}[htbp]
\begin{center}
\includegraphics[scale=0.5]{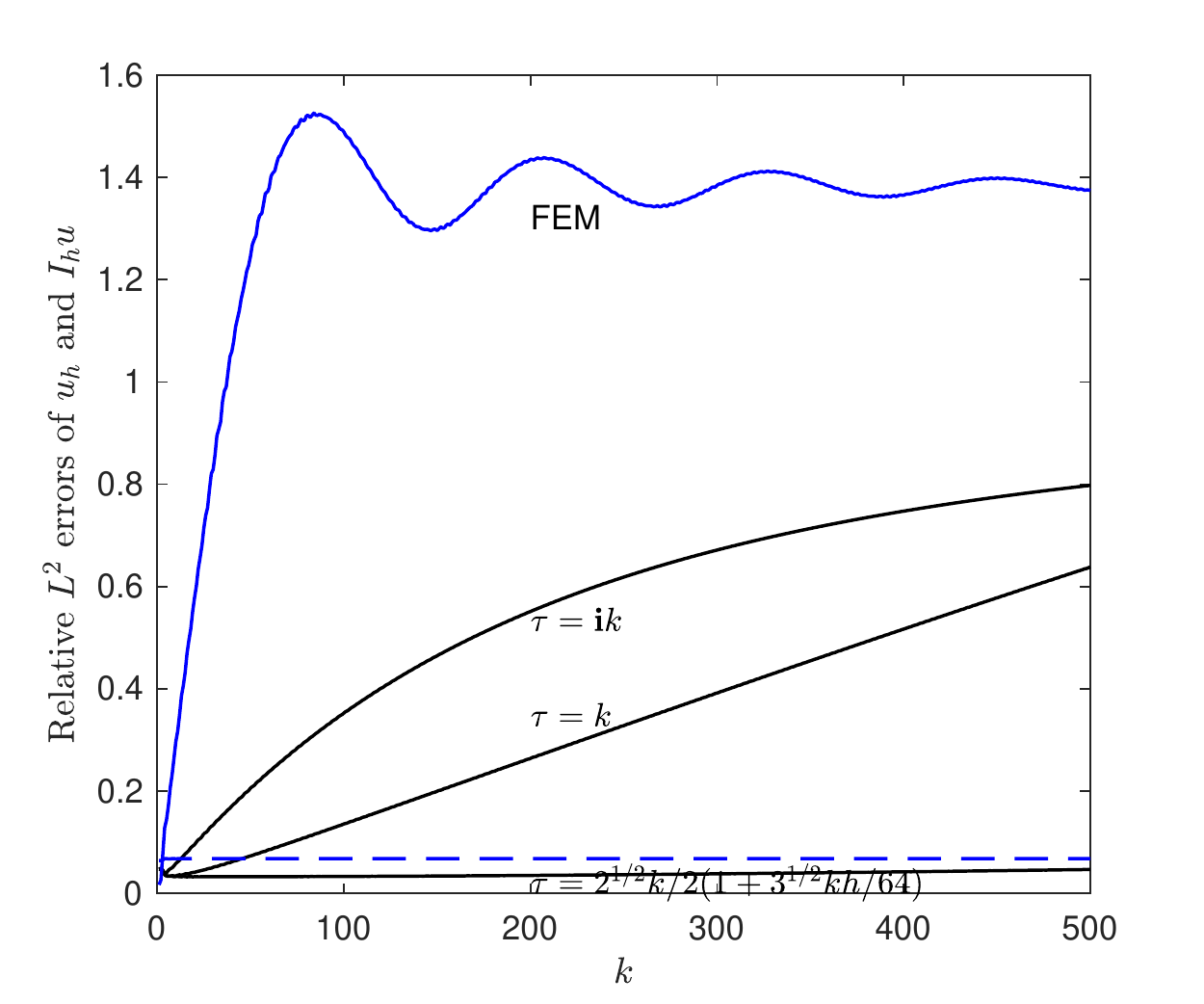}
\includegraphics[scale=0.5]{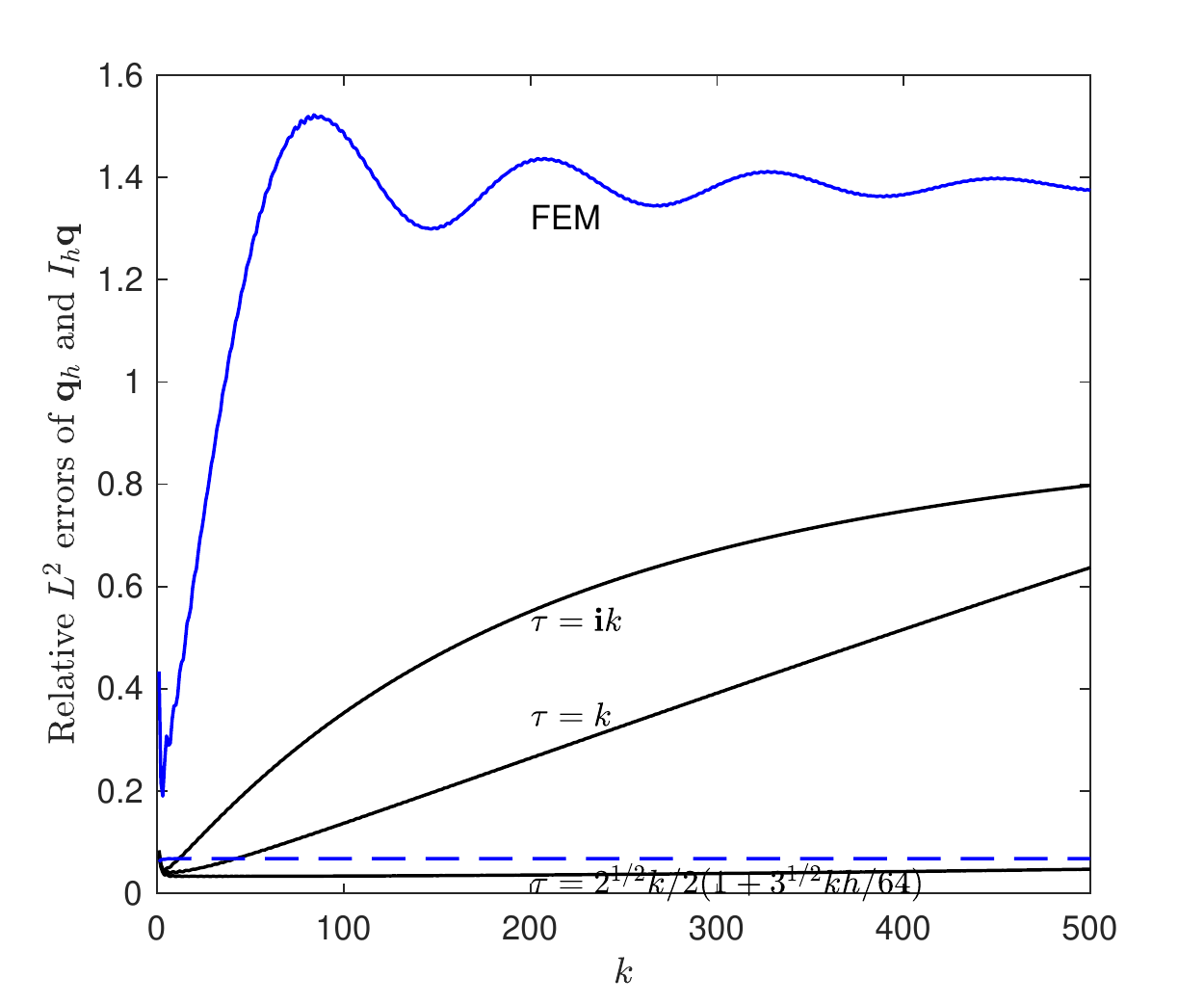}
\end{center}
\caption{Example  \ref{ex2}: $kh=1$.  
 Relative $L^2$ errors $e_u$ (left) and $e_\bq$ (right) versus $k=1,2,\cdots,500$ for $\tau=\mathbf{i}k, k$, and $\frac{\sqrt{2}}{2}k\big(1+\frac{\sqrt{3}}{64}kh\big)$, respectively. The dashed lines plot the relative interpolation errors.\label{F8}}
\end{figure}


\end{document}